\documentclass[11pt,twoside]{article}
\usepackage{amsmath,amssymb,amsthm}
\usepackage{amsmath,amssymb,amsthm}
\usepackage{graphicx}
\usepackage{epstopdf}
\usepackage{caption}
\allowdisplaybreaks[4]

\newtheorem{lm}{Lemma}[section]
\newtheorem{thm}{Theorem}[section]

\newcounter{saveeqn}%

\makeatletter
\evensidemargin
\oddsidemargin
\makeatother

\makeatletter
\@addtoreset{equation}{section}
\makeatother

\title{\Large\bf Periodic orbits of linear Filippov systems with a line of discontinuity
\thanks{
Supported by NSFC 11471228, Graduate Student's Research and Innovation Fund of Sichuan University 2018YJSY047.
}
}

\author{Tao Li, ~~~~Xingwu Chen\!\!
\footnote{Author to whom any correspondence should be addressed. Email address: xingwu.chen@hotmail.com (X. Chen).}
\\
{\small School of Mathematics, Sichuan University,}
{\small Chengdu, Sichuan 610064, P. R. China}
}
\date{}

\begin{document}
\maketitle

\begin{abstract}
In this paper we consider periodic orbits of planar linear Filippov systems with a
line of discontinuity. Unlike many publications researching only the maximum number of crossing periodic orbits,
we investigate not only the number and configuration of sliding periodic orbits, but also
the coexistence of sliding periodic orbits and crossing ones. Firstly, we prove that the number of sliding periodic orbits is at most 2, and
give all possible configurations of one or two sliding periodic orbits. Secondly,
we prove that two sliding periodic orbits coexist with at most one crossing periodic orbit, and one
sliding periodic orbit can coexist with two crossing ones.
\end{abstract}

\baselineskip 15pt
\parskip 10pt
\thispagestyle{empty}
\setcounter{page}{1}

\section{Introduction and main results}

\setcounter{equation}{0}
\setcounter{lm}{0}
\setcounter{thm}{0}
\setcounter{rmk}{0}
\setcounter{df}{0}
\setcounter{cor}{0}

An interesting class in nonsmooth dynamical systems is so-called Filippov systems, sometimes called as piecewise smooth
discontinuous systems. It is used as a mathematical model in many fields
such as feedback systems in control systems \cite{ChL-FG, ChL-GF1},
power electronics \cite{MBCJ}, oscillators and dry frictions in mechanical engineering \cite{HCJL, PK} and so on.
Due to the switching surface, there are many novel
dynamical behaviors which do not exist in smooth systems, such as grazing and sliding motion
(see \cite{ChL-MG, ChL-YA}). The research of periodic orbits is one of important and challenging topics in
both smooth systems and Filippov systems. However, in Filippov systems there are
three classes of new periodic orbits which do not appear in smooth systems because of switching surfaces.
They are so-called crossing periodic orbits, grazing periodic orbits and sliding periodic orbits(see, e.g., \cite{ChL-DB, ChL-YA}).

Many models in applications are linear Filippov systems with a line of discontinuity, such as
the direct voltage control system of the buck converter(see \cite{CO, ChL-EE})
$$
\begin{aligned}
\left(
\begin{array}{c}
\dot x\\
\dot y
\end{array}
\right)=
\left(
\begin{array}{cc}
-a &1\\
-1 &-b
\end{array}
\right)
\left(
\begin{array}{c}
x\\
y
\end{array}
\right)+
\left\{
\begin{aligned}
&\left(
\begin{array}{c}
0\\
0
\end{array}
\right) &&{\rm if}~x>x_{ref},\\
& \left(
\begin{array}{c}
0\\
1
\end{array}
\right) &&{\rm if}~x<x_{ref},
\end{aligned}
\right.
\end{aligned}
$$
where $a>0, b>0$ are normalized parameters of the circuit, $x, y$ are the normalized load
voltage and impedance current, $x_{ref}$ is the desired normalized load voltage.
Another system
$$\ddot x+\alpha\dot x+\mu x{\rm sgn}\dot x+\beta x=0$$
models an unforced mechanical oscillator with dry friction $\mu x{\rm sgn}\dot x$ (see \cite{LT}),
where $\alpha\dot x$ and $\beta x$ denote respectively the viscous damping and restoring force.
This system is also written as a linear  Filippov system
$$
\begin{aligned}
\left(
\begin{array}{c}
\dot x\\
\dot y
\end{array}
\right)=
\left\{
\begin{aligned}
&\left(
\begin{array}{cc}
0 &1\\
-\beta-\mu &-\alpha
\end{array}
\right)\left(
\begin{array}{c}
x\\
y
\end{array}
\right) &&{\rm if}~y>0,\\
& \left(
\begin{array}{cc}
0 &1\\
-\beta+\mu &-\alpha
\end{array}
\right)\left(
\begin{array}{c}
x\\
y
\end{array}
\right) &&{\rm if}~y<0.
\end{aligned}
\right.
\end{aligned}
$$
Thus, the investigation of linear Filippov systems is one of important topics in the field of nonsmooth dynamical systems.

Generally, a linear Filippov system of dimension $2$ with a line of discontinuity is of form
\begin{eqnarray}
\left\{
\begin{aligned}
  \dot z&=A^+z+b^+~~~~&&{\rm if~}z\in\Sigma^+,\\
  \dot z&=A^-z+b^-~ ~&&{\rm if~}z\in\Sigma^-,
\end{aligned}
\right.
\label{PFS}
\end{eqnarray}
where $z:=(x, y)^\top\in\mathbb{R}^2$, $\Sigma^+:=\{z\in\mathbb{R}^2: H(z)>0\}, \Sigma^-:=\{z\in\mathbb{R}^2: H(z)<0\}$,
\begin{eqnarray*}
\begin{aligned}
&A^\pm:=\left(
\begin{array}{cc}
  a^\pm_{11}&a^\pm_{12}\\
  a^\pm_{21}&a^\pm_{22}\\
\end{array}
\right)\in \mathbb{R}^{2\times 2},~~~~b^\pm:=\left(
\begin{array}{c}
  b^\pm_1\\
  b^\pm_2\\
\end{array}
\right)\in \mathbb{R}^2.
\end{aligned}
\end{eqnarray*}
Here $H(z):=c\cdot z+d$, $c=(c_1,c_2)\in\mathbb{R}^2\backslash\{(0,0)\}$, $d\in\mathbb{R}$.
Let
$$\Sigma:=\{z\in\mathbb{R}^2: H(z)=0\}.$$
Sometimes $\Sigma$ is called the {\it switching line} or {\it discontinuity line} of system (\ref{PFS}).
For system (\ref{PFS}), solutions without points in $\Sigma$ are naturally determined by the vector filed $A^+z+b^+$ or $A^-z+b^-$. However, if a solution reaches $\Sigma$ at a time, then a new rule must be adopted to define its evolution. A widely use method is the so-called Filippov convention (see \cite{ChL-AF, ChL-YA}). In particular, as seen in \cite{ChL-YA}, $\Sigma$ is divided into the {\it crossing region}
$$\Sigma^c:=\{z\in \Sigma: ~(c\cdot(A^+z+b^+))(c\cdot(A^-z+b^-))>0\}$$
and the {\it sliding region}
$$\Sigma^s:=\{z\in\Sigma: ~(c\cdot(A^+z+b^+))(c\cdot(A^-z+b^-))\le0\}.$$
In addition,
$$
\begin{aligned}
\Sigma^s_a:&=\{z\in \Sigma^s:~c\cdot(A^+z+b^+)<0,~ c\cdot(A^-z+b^-)>0\},\\
\Sigma^s_r:&=\{z\in \Sigma^s:~c\cdot(A^+z+b^+)>0,~ c\cdot(A^-z+b^-)<0\}
\end{aligned}
$$
are called {\it attractive siding region} and {\it repulsive sliding region}, respectively.
In $\Sigma^c$, both two vector fields are transversal to $\Sigma$ and their normal components have same sign.
Thus the solution passing through a point in $\Sigma^c$ crosses $\Sigma$ at the point. In $\Sigma^s$, either their normal components have opposite sign or at least one of them vanishes. In this case, there exists the so-called {\it sliding solution}, which is the flow of a differential equation
\begin{eqnarray}
\dot z=F^s(z)~~~~~~~{\rm for}~~ z\in\Sigma^s.
\label{sliding}
\end{eqnarray}
For $z\in\Sigma^s$ satisfying $c\cdot(A^-z+b^-)\ne c\cdot(A^+z+b^+)$, $F^s(z)$ is given by
\begin{eqnarray}
F^s(z)=
\frac{c\cdot(A^-z+b^-)(A^+z+b^+)-c\cdot(A^+z+b^+)(A^-z+b^-)}
{c\cdot(A^-z+b^-)-c\cdot(A^+z+b^+)}
\label{sliding0}
\end{eqnarray}
However, for $z\in\Sigma^s$ satisfying $c\cdot(A^-z+b^-)=c\cdot(A^+z+b^+)$, which termed as a {\it singular sliding point},
$F^s(z)$ either is defined by extending (\ref{sliding0}) if an extension is possible or is defined by $0$ (see \cite{ChL-YA}).
Usually, $F^s(z)$ is called the {\it sliding vector field} and its equilibria are called {\it pseudo-equilibria} of system (\ref{PFS}).
In conclusion, the solution of system (\ref{PFS}) can be defined by concatenating the flows of $A^+z+b^+, A^-z+b^-$ and $F^s(z)$.
A precise description is given in \cite[p.2160]{ChL-YA}, where both forward and backward solutions are defined uniquely. Although that,
the invertibility in the classical sense does not hold for system (\ref{PFS}) because its orbits can overlap.

Besides, from \cite{ChL-YA} a point in the boundary of $\Sigma^s$ is either a {\it boundary equilibrium} where one of the vector fields $A^+z+b^+$ and $A^-z+b^-$ vanishes, or a {\it tangency point} where $A^+z+b^+$ and $A^-z+b^-$ do not vanish and at least one of them is tangent to $\Sigma$. Note that
if a tangency point is of $A^+z+b^+$ and $A^-z+b^-$, then it is a singular sliding point. Moreover, a tangency point $q$
of the vector field $A^+z+b^+$ is {\it visible} (resp. {\it invisible}) if the orbit of $A^+z+b^+$ passing through $q$ at time $t_q$ stays in $\overline{\Sigma^+}$ (resp. $\overline{\Sigma^-}$) for small $|t-t_q|$. An equilibrium of $A^+z+b^+$ is {\it admissible} (resp. {\it virtual}) if it lies in $\Sigma^+$ (resp. $\Sigma^-$). For the vector field $A^-z+b^-$, the visibility of tangency points, the admissibility and virtuality of equilibria can be defined similarly.

According to \cite{ChL-YA}, a periodic orbit lying entirely in $\Sigma^+$ or $\Sigma^-$ is said to be a {\it standard periodic orbit}. Moreover, a periodic orbit is said to be a {\it sliding periodic orbit} if it has a sliding segment in $\Sigma$, and {\it crossing periodic orbit} if it has only isolated points in $\Sigma$. Notice that a sliding periodic orbit forward (resp. backward) in time can be not
a periodic orbit backward (resp. forward) in time, because the invertibility in the classical sense does not hold for system (\ref{PFS}) as indicated above.

The investigation of periodic orbits of system (\ref{PFS}) can be traced back to 1930s (see \cite{ChL-AA}).
When $c_1=1$, $c_2<0$, $d=0$ and
$$A^+=A^-=\left(
\begin{array}{cc}
  0&1\\
  0&-1\\
\end{array}
\right),~~~~~~
b^+=-b^-=\left(
\begin{array}{c}
  0\\
  -1\\
\end{array}
\right),$$
the existence of crossing periodic orbits was proved in \cite{ChL-AA}. When $c_1=1, d=0$ and
$$A^+=A^-=\left(
\begin{array}{cc}
  0&1\\
  -a_1&-2a_2\\
\end{array}
\right),~~~~~~
b^+=-b^-=\left(
\begin{array}{c}
  0\\
  -b_2\\
\end{array}
\right)$$
with positive $a_1, a_2, b_2$, it was proved in \cite{ChL-BH} that there exist (resp. exists no) asymptotically
stable crossing periodic orbits if $c_2<0$ (resp. $c_2\ge0$).
As to sliding periodic orbits, the existence and number were researched under the condition $A^+=A^-$ and $b^+=-b^-$ in
\cite{ChL-FG, ChL-GF1, ChL-PK}. It was proved in \cite{ChL-PK} that (\ref{PFS}) has no sliding periodic orbits
if the two eigenvalues of $A^+$ are either real or pure imaginary.
The number of sliding periodic orbits and the number of crossing periodic orbits were proved to be at most
$2$ separately in \cite{ChL-FG} when $0<({\rm tr(A^+)})^2<4{\rm det} A^+$. Moreover,
the number of crossing periodic orbits is at most $1$ when there are two sliding periodic orbits.
It was proved in \cite{ChL-GF1} that (\ref{PFS}) has no sliding periodic orbits
and at most one crossing periodic orbit if ${\rm det} A^+<0$.

By the transformation $z\to B(z+\nu)$, switching line $\Sigma$ can be transformed as $y$-axis, where $\nu=(-d, 0)^\top$ and
$$
B=\left\{
\begin{aligned}
&\left(
\begin{array}{cc}
  1/c_1&-c_2/c_1\\
  0&1\\
\end{array}
\right)~~~~~&&{\rm if}~~ c_1\ne0,\\
&\left(
\begin{array}{cc}
  0&~1\\
  1/c_2&0\\
\end{array}
\right)~~~~~&&{\rm if} ~~c_1=0.
\end{aligned}
\right.$$
Thus, without loss of generality we always consider system (\ref{PFS}) with $y$-axis as the switching line, i.e.,
\begin{eqnarray}
\left\{
\begin{aligned}
  \dot z&=A^+z+b^+~~~~&&{\rm if~}x>0,\\
  \dot z&=A^-z+b^-~ ~&&{\rm if~}x<0.
\end{aligned}
\right.
\label{DPWL}
\end{eqnarray}
We say that system (\ref{DPWL}) is {\it nondegenerate} if $A^{\pm}$ are both nondegenerate.
Besides, $\dot z=A^+z+b^+$ (resp. $\dot z=A^-z+b^-$) is called the {\it left system} (resp. {\it right system}) of
(\ref{DPWL}). For nondegenerate system (\ref{DPWL}), we denote the unique equilibrium of the left system
(resp. right system) by $E_L$ (resp. $E_R$).

Since standard periodic orbits of system (\ref{DPWL}) appear only when either $E_L$ or $E_R$ is a center,
it is trivial to study standard periodic orbits.
Thus, in this paper we consider the crossing periodic orbits and sliding ones of system (\ref{DPWL}).
When system (\ref{DPWL}) satisfies
$$A^\pm=\left(
\begin{array}{cc}
  0&1\\
  a^\pm_1&a^\pm_2\\
\end{array}
\right),~~~~a^\pm_1>0,~~~~~
b^+=-b^-=\left(
\begin{array}{c}
  b_1\\
  -b_2\\
\end{array}
\right),$$
it was proved in \cite{ChL-SS} that the numbers of sliding periodic orbits and crossing ones are
$0$ and at least $1$ respectively if $|2b_2+(a^+_2+a^-_2)b_1|/|a^+_2-a^-_2|>|b_1|$.
In recent years, many researchers are interested in the maximum number of crossing periodic orbits of system (\ref{DPWL}).
In 2010, it was conjectured in \cite{ChL-HZ} that system (\ref{DPWL}) has at most two crossing periodic orbits.
A negative answer to this conjecture was given in 2012
by an example with three crossing periodic orbits in \cite{ChL-SX} and lately more systems with three crossing periodic orbits
were found in \cite{ChL-DC, ChL-CCJ, ChL-EE1, ChL-FEEE, ChL-SX, ChL-LLP, ChL-JLDM, ChL-JLDM1, ChL-LE}.
But the maximal number of crossing periodic orbits of (\ref{DPWL}) is still unknown. More results on crossing periodic orbits
refer to \cite{ChL-EE, ChL-SX1, ChL-SX2, ChL-ZXL, ChL-HLH1, ChL-HLH}. On the other hand, for system (\ref{DPWL}) we lack the information about
sliding periodic orbits and an interesting problem is {\it what is the maximum number of sliding periodic orbits
and their configuration}. Another interesting problem is {\it what about the coexistence of sliding periodic
orbits and crossing ones}.

Motivated by these interesting problems, in this paper we study the number and configuration of sliding periodic orbits
of system (\ref{DPWL}), the coexistence of sliding periodic orbits and crossing ones.
As introduced above, these problems have been researched in \cite{ChL-FG, ChL-GF1, ChL-PK} when $A^+=A^-$ and $b^+=-b^-$.
Therefore, in present paper we consider the more general system (\ref{DPWL}).
The following theorems are our main results.

\begin{thm}
The number of sliding periodic orbits for nondegenerate system {\rm(\ref{DPWL})} is at most $2$. In particular,
\begin{description}
\setlength{\itemsep}{0mm}
\item[]{\rm(i)} if {\rm(\ref{DPWL})} has a unique sliding periodic orbit, then the configuration of this orbit is one of {\rm Figure \ref{1A}(a)-(d)} in the sense of $\Sigma$-equivalence and time reversing;
\item[]{\rm(ii)} if {\rm(\ref{DPWL})} has exactly $2$ sliding periodic orbits, then the configuration of these two orbits is one of {\rm Figure \ref{2A}(a)-(c)} in the sense of $\Sigma$-equivalence and time reversing.
\end{description}
\label{typeAnumber}
\end{thm}

\begin{figure}
  \begin{minipage}[t]{0.20\linewidth}
  \centering
  \includegraphics[width=1.48in]{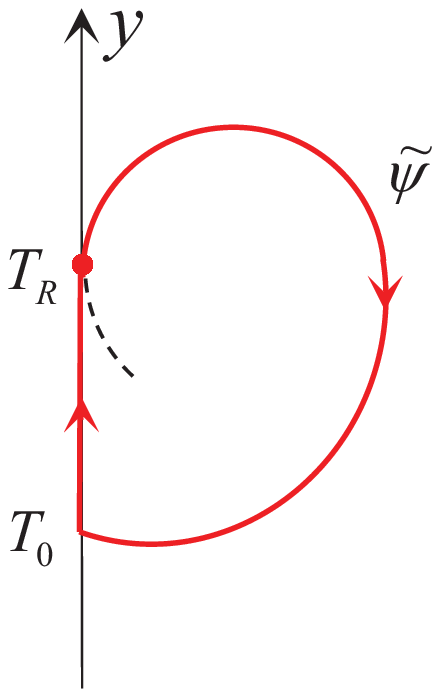}
  \caption*{~(a)}
  \end{minipage}
  ~~~~~~~~
  \begin{minipage}[t]{0.20\linewidth}
  \centering
  \includegraphics[width=1.50in]{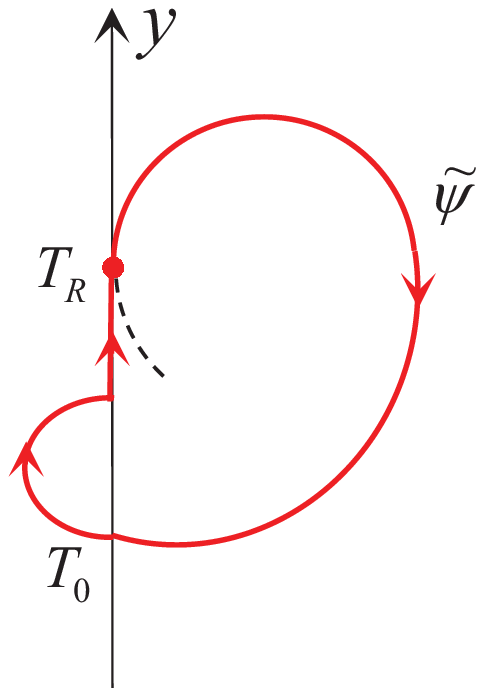}
  \caption*{~(b)}
  \end{minipage}
  ~~~~~~~~
  \begin{minipage}[t]{0.20\linewidth}
  \centering
  \includegraphics[width=1.50in]{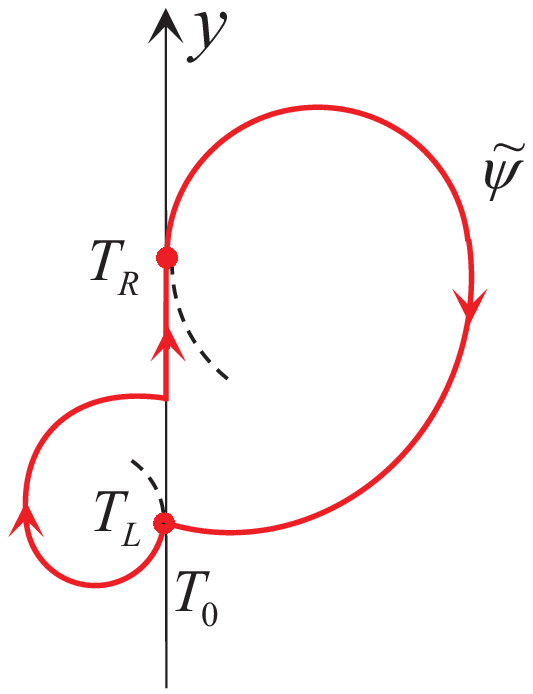}
  \caption*{~~(c)}
  \end{minipage}
  ~~~~~~~~~
  \begin{minipage}[t]{0.20\linewidth}
  \centering
  \includegraphics[width=1.35in]{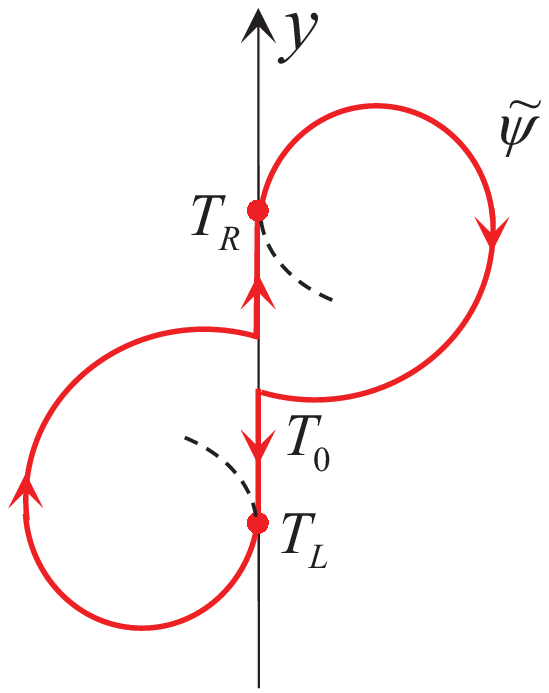}
  \caption*{~(d)}
  \end{minipage}
\caption{{\small Configurations of one sliding periodic orbit}}
\label{1A}
\end{figure}

\begin{figure}
~~~~~~~~
  \begin{minipage}[t]{0.25\linewidth}
  \centering
  \includegraphics[width=1.485in]{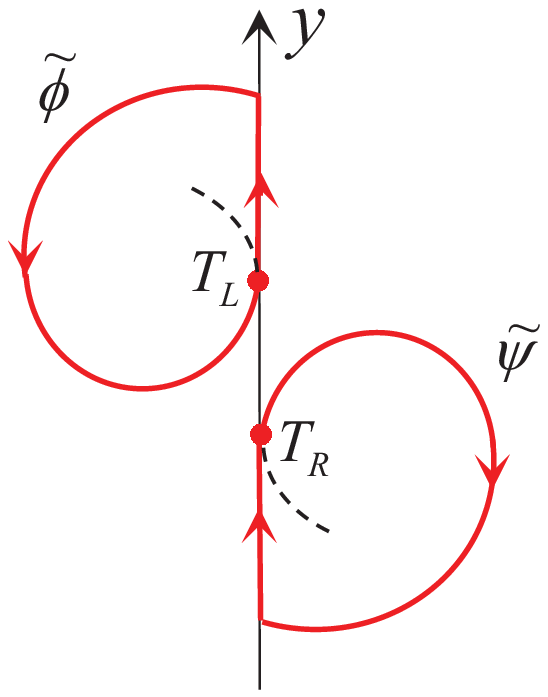}
  \caption*{(a)}
  \end{minipage}
  \~~~~~~~~~
  \begin{minipage}[t]{0.25\linewidth}
  \centering
  \includegraphics[width=1.420in]{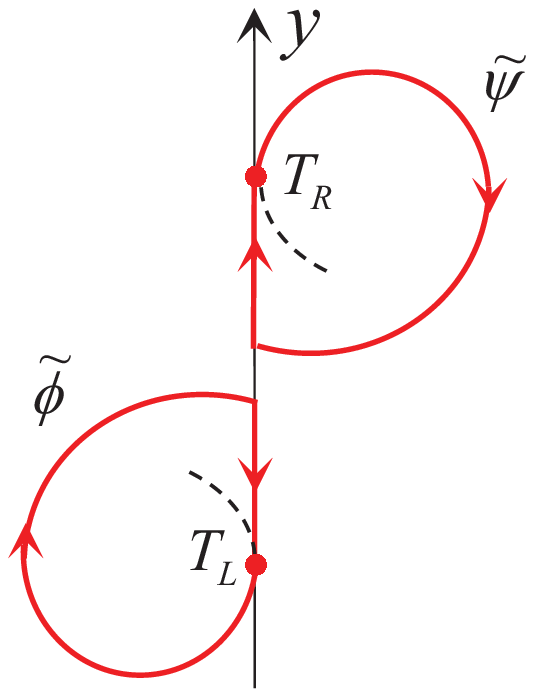}
  \caption*{(b)}
  \end{minipage}
   ~~~~~~~
  \begin{minipage}[t]{0.25\linewidth}
  \centering
  \includegraphics[width=1.50in]{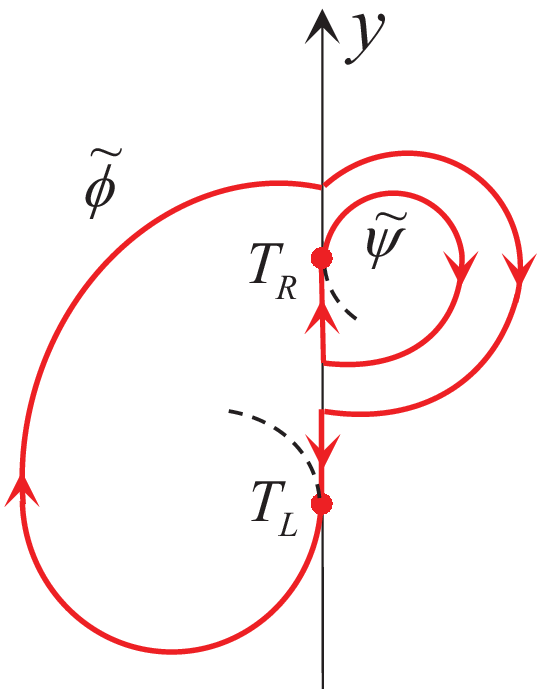}
  \caption*{(c)}
  \end{minipage}
\caption{{\small Configurations of two sliding periodic orbits}}
\label{2A}
\end{figure}

For nondegenerate system (\ref{DPWL}), we give the number and configuration of sliding periodic orbits in Theorem \ref{typeAnumber}.
The existence of all configurations for sliding periodic orbits are shown by examples in Section 4.

Two Filippov systems are {\it $\Sigma$-equivalent} \cite[Definition 2.20]{ChL-MG} if there exists an orientation preserving homeomorphism that maps the orbits and switching boundaries of the first system onto orbits and switching boundaries of the second one. Therefore, Theorem~\ref{typeAnumber} means that by an orientation preserving homeomorphism and time reversing(if necessary), system (\ref{DPWL}) with one (resp. two) sliding periodic orbit can be changed into a system of form (\ref{DPWL}) having one (resp. two) sliding periodic orbit shown in one of {\rm Figure \ref{1A}(a)-(d)} (resp. Figure \ref{2A}(a)-(c)).

Note that system (\ref{DPWL}) with a sliding periodic orbit shown in Figure~\ref{1A}(c) is not structurally stable and
undergoes the so-called {\it simple sliding bifurcation} (see \cite{ChL-YA}) under perturbation, i.e., the sliding periodic orbit shown in Figure~\ref{1A}(c) is replaced by a sliding periodic orbit shown in Figure~\ref{1A}(b) or Figure~\ref{1A}(d).

In order to study the total number of periodic orbits, we give the following theorems on the coexistence of
sliding periodic orbits and crossing ones.

\begin{thm}
For  system {\rm(\ref{DPWL})}, the following statements hold.
\begin{description}
\setlength{\itemsep}{0mm}
\item[]{\rm(i)} If {\rm(\ref{DPWL})} has two sliding periodic orbits shown in {\rm Figure~\ref{2A}(a)},
then there exist no crossing periodic orbits.
\item[]{\rm (ii)} If {\rm(\ref{DPWL})} has either two sliding periodic orbits shown in one of {\rm Figure~\ref{2A}(b)(c)} or a unique sliding periodic orbit shown in one of {\rm Figure~\ref{1A}(c)(d)}, then there exists a unique crossing periodic orbit, which is unstable.
\item[]{\rm(iii)} There exists a system of form {\rm (\ref{DPWL})} having one sliding periodic orbit shown in {\rm Figure~\ref{1A}(a)} and
 two crossing periodic orbits.
 \item[]{\rm(iv)}There exists a system of form {\rm (\ref{DPWL})} having one sliding periodic orbit shown in {\rm Figure~\ref{1A}(b)} and
 two crossing periodic orbits.
 \end{description}
\label{AC}
\end{thm}

This paper is organized as follows. In Section 2, after introducing a canonical form of system (\ref{DPWL})
which captures the sliding periodic orbits and crossing ones, we provide the proof of Theorem~\ref{typeAnumber}. In Section 3, we give the proof of Theorem~\ref{AC}. Lastly, some concluding remarks are given in Section 4 to end this paper.

\section{Proof of Theorem~\ref{typeAnumber}}
\setcounter{equation}{0}
\setcounter{lm}{0}
\setcounter{thm}{0}
\setcounter{rmk}{0}
\setcounter{df}{0}
\setcounter{cor}{0}

In order to prove Theorem~\ref{typeAnumber}, we give some lemmas in the following.
\begin{lm}
Assume that $(x(t), y(t))^\top$ is the unique solution of the Cauchy problem
\begin{eqnarray}
\left(
\begin{array}{c}
\dot x\\
\dot y\\
\end{array}
\right)
=\left(
\begin{array}{cc}
  a_{11}&a_{12}\\
  a_{21}&a_{22}\\
\end{array}
\right)
\left(
\begin{array}{c}
x\\
y\\
\end{array}
\right)+
\left(
\begin{array}{c}
  b_1\\
  b_2\\
\end{array}
\right)
\label{gfkl}
\end{eqnarray}
with $(x(0), y(0))^\top=(0, -b_1/a_{12})^\top$, where $a_{12}\neq0, a_{11}a_{22}-a_{12}a_{21}\neq0, b_2a_{12}-a_{22}b_1\neq0$.
\\
{\rm (i)} There exists a $t_0\ne 0$ such that $x(t_0)=0, y(t_0)=-b_1/a_{12}$ if and only if the unique equilibrium of
{\rm (\ref{gfkl})} is a center;
\\
{\rm (ii)} There exists a $t_0>0$ {\rm(}resp. $<0${\rm)} such that $x(t_0)=0, y(t_0)\ne -b_1/a_{12}$ if and only if the unique equilibrium of
{\rm (\ref{gfkl})} is an unstable {\rm(}resp. a stable{\rm)} focus.
\label{lmmm}
\end{lm}

The proof of Lemma~\ref{lmmm} is elementary and it is neglected.

\begin{lm}
If nondegenerate system {\rm(\ref{DPWL})} has a sliding periodic orbit $\psi$, then $a^+_{12}b_1^-\ne a^-_{12}b^+_1$ and
at least one of $E_L$ and $E_R$ is an admissible focus.
Additionally, the admissible focus is unstable {\rm(}resp. stable{\rm)} when $\psi\cap\Sigma^s_r=\emptyset$
{\rm(}resp. $\psi\cap\Sigma^s_a=\emptyset${\rm)}.
\label{suff}
\end{lm}

\begin{proof}
Suppose that $a^+_{12}b_1^-=a^-_{12}b^+_1$. Since $\psi$ is a sliding periodic orbit of (\ref{DPWL}), $\Sigma^s_a\cup\Sigma^s_r\ne \emptyset$ by the definition of sliding periodic orbit given in Section 1. Thus we get either $a^+_{12}=a^-_{12}=0$, $b_1^+b_1^+<0$ or $a^+_{12}a^-_{12}<0$ by the definitions of $\Sigma^s_a$ and $\Sigma^s_r$. In the former case,
$\Sigma=\Sigma^s_a$ for $b_1^+<0$ and $\Sigma=\Sigma^s_r$ for $b_1^+>0$, while in the latter case, $\Sigma$ consists of $\Sigma^s_a$, $\Sigma^s_r$ and a singular sliding point. On the other hand, as introduced in Section 1, we adopt the definition of solutions proposed in \cite{ChL-YA}. Thus any orbit reaching $\Sigma$ at a time stays in $\Sigma$ forever for $\Sigma=\Sigma^s$, which contradicts that $\psi$ is a sliding periodic orbit. Hence, $a^+_{12}b_1^-\ne a^-_{12}b^+_1$ and then there exist no singular sliding points, provided the existence of sliding periodic orbits. Furthermore,
$\psi$ cannot slide simultaneously on $\Sigma^s_a$ and $\Sigma^s_r$. Otherwise, the sliding orbit of $\psi$ must be from $\Sigma^s_a$ (resp. $\Sigma^s_r$) to $\Sigma^s_r$ (resp. $\Sigma^s_a$) after going through a singular sliding point by the linearity of the left and right systems.

In the case that $\psi\cap\Sigma^s_r=\emptyset$, i.e., $\psi$ slides only on a segment of $\Sigma^s_a$,
$\psi$ either leaves $\Sigma$ at a visible tangency point of the left system and enters into $\Sigma^-$ or
leaves $\Sigma$ at a visible tangency point of the right system and enters into $\Sigma^+$ as $t$ increases.
Without loss of generality, we assume that
$\psi$ leaves $\Sigma$ at a visible tangency point $q$ of the left system and enters into $\Sigma^-$.
Associated with the definition of tangency point given in Section 1,
$b_1^-=0$ if $a_{12}^-=0$ and, in such case, $\Sigma^s=\emptyset$.
Thus, $a_{12}^-\ne 0$. Clearly, $q$ lies at $(0, -b_1^-/a_{12}^-)^\top$.
Due to the sliding motion, $\psi$ reaches
$\Sigma$ again at a different point $p$ from $q$ after a finite time.
Let the ordinate of $p$ be $(0, \alpha_0)^\top$, where $\alpha_0$ is a constant different from $-b_1^-/a_{12}^-$.
Thus, there exists a $t_0>0$ such that the solution $(x(t), y(t))^\top$
of the left system of (\ref{DPWL}) with $(x(0), y(0))^\top=(0, -b_1^-/a_{12}^-)^\top$
satisfies $x(t_0)=0, ~y(t_0)=\alpha_0$. By (ii) of Lemma~\ref{lmmm}, equilibrium $E_L$ in (\ref{DPWL}) is an unstable focus.
Denote the region surrounded by $\Sigma$ and the orbit from $q$ to $p$ in the left plane by $\Xi$.
For the left system, in $\Xi$ there exists an equilibrium by the Poincar\'e-Bendixson Theorem (see \cite[p. 54]{ChL-JK}).
Thus, $E_L$ is admissible by the definition of admissible equilibrium given in Section 1.

In the case that $\psi\cap\Sigma^s_a=\emptyset$, i.e., $\psi$ slides only on a segment of $\Sigma^s_r$, by time reversing
$\psi$ becomes a sliding periodic orbit sliding only on $\Sigma^s_a$. Moreover, the stabilities of $E_L$ and $E_R$
change but, the type and admissibility of these two equilibria do not change. Associated with the result given
in last paragraph, at least one of $E_L$ and $E_R$ is a stable admissible focus.
\end{proof}

Due to many parameters, it is necessary to reduce system (\ref{DPWL}) to a canonical form with less parameters. A celebrated result is \cite[Proposition 3.1]{ChL-EE}, where a canonical form with only $7$ parameters and capturing crossing periodic orbits is obtained. However, as indicated in that paper, the obtained canonical form cannot capture sliding periodic orbits, because the used change of variables is only continuous. Therefore, based on our purpose, we present a new canonical form that is $\Sigma$-equivalent to system (\ref{DPWL}) in the following Lemma~\ref{norm1}. This means that
a crossing (resp. sliding) periodic orbit of system (\ref{DPWL}) is transformed into a crossing (resp. sliding) periodic orbit of the canonical form, vice versa.

\begin{lm}
If $a^+_{12}b_1^-\ne a^-_{12}b^+_1$ and at least one of $E_L$ and $E_R$
is an admissible focus, then nondegenerate system {\rm(\ref{DPWL})} is $\Sigma$-equivalent to the canonical form
\begin{eqnarray}
\left\{
\begin{aligned}
\dot z&=\left(
\begin{array}{cc}
  2\alpha&1\\
  -1-\alpha^2&0\\
\end{array}
\right)z+
\left(
\begin{array}{c}
  0\\
  \beta\\
\end{array}
\right)~~~~{\rm if}~x>0,\\
\dot z&=\left(
\begin{array}{cc}
  ~~\gamma_1&~~\delta~~\\
  ~~\gamma_2&~~\gamma_3~~\\
\end{array}
\right)z+
\left(
\begin{array}{c}
  \eta\\
  \rho\\
\end{array}
\right)~~~~~{\rm if}~x<0
\end{aligned}
\right.
\label{dsgsh}
\end{eqnarray}
with $\alpha\ne 0, \beta>0, \eta\ne0$. Additionally, $\alpha>0$ {\rm(}resp. $\alpha<0${\rm)} if this admissible focus is unstable {\rm(}resp. stable{\rm)}.\\
\label{norm1}
\end{lm}

\begin{proof}
Since the right system and the left one exchange under the transformation $x\to -x$,
we assume that $E_R$ is an admissible focus. Consequently, $a_{12}^+\ne 0$.

By the change
\begin{eqnarray}
z\to \left(
\begin{array}{cc}
  1&0\\
  a_{22}^+/a_{12}^+&a_{12}^+\\
\end{array}
\right)z+
\left(
\begin{array}{c}
  0\\
  -b_1^+/a_{12}^+\\
\end{array}
\right),
\label{vch}
\end{eqnarray}
system (\ref{DPWL}) is transformed into
\begin{eqnarray}
\left\{
\begin{aligned}
\dot z&=\left(
\begin{array}{cc}
  A_{11}&(a_{12}^+)^2\\
  A_{21}&0\\
\end{array}
\right)z+
\left(
\begin{array}{c}
  ~0~\\
  ~B~\\
\end{array}
\right)~~~~{\rm if}~x>0,\\
\dot z&=\left(
\begin{array}{cc}
  C_{11}&a_{12}^+a_{12}^-\\
  C_{21}&C_{22}\\
\end{array}
\right)z+
\left(
\begin{array}{c}
  D_1\\
  D_2\\
\end{array}
\right)~~~~~{\rm if}~x<0,
\end{aligned}
\right.
\label{mmmm}
\end{eqnarray}
and
\begin{eqnarray}
\begin{aligned}
&A_{11}=a^+_{11}+a_{22}^+, ~~~A_{21}=\frac{a_{12}^+a^+_{21}-a_{11}^+a_{22}^+}{(a_{12}^+)^2},~~~ B=\frac{a_{12}^+b_2^+-a_{22}^+b_1^+}{(a_{12}^+)^2},\\
&C_{11}=\frac{a^-_{11}a_{12}^++a_{12}^-a_{22}^+}{a_{12}^+},~~~C_{22}=\frac{a_{12}^+a_{22}^--a_{22}^+a_{12}^-}{a_{12}^+},
~~~D_1=\frac{a_{12}^+b_1^--a_{12}^-b_1^+}{a_{12}^+},\\
&C_{21}=\frac{a_{12}^+a^-_{21}+a_{22}^-a_{22}^+-a_{22}^+C_{11}}{(a_{12}^+)^2}, ~~~D_2=\frac{a_{12}^+b_2^--a_{22}^+b_1^--b_1^+C_{22}}{(a_{12}^+)^2}.
\end{aligned}
\label{ABD}
\end{eqnarray}
Note that the change (\ref{vch}) does not change the switching line.
Using time rescaling $t\to t/(a_{12}^+)^2$ for the right system of (\ref{mmmm}) and $t\to ut$ for the left system of (\ref{mmmm}),
we get
\begin{eqnarray}
\left\{
\begin{aligned}
\dot z&=\left(
\begin{array}{cc}
  A_{11}/(a^+_{12})^2&1\\
  A_{21}/(a_{12}^+)^2&0\\
\end{array}
\right)z+
\left(
\begin{array}{c}
  0\\
  B/(a_{12}^+)^2\\
\end{array}
\right)~~~~{\rm if}~x>0,\\
\dot z&=\left(
\begin{array}{cc}
 ~uC_{11}~&\delta~\\
  ~uC_{21}~&uC_{22}~\\
\end{array}
\right)z+
\left(
\begin{array}{c}
  ~~~uD_1~~~\\
  ~~~uD_2~~~\\
\end{array}
\right)~~~~{\rm if}~x<0,
\end{aligned}
\right.
\label{sflko}
\end{eqnarray}
where $\delta={\rm sgn}(a_{12}^-a_{12}^+)$ and
$$
u=\left\{
\begin{aligned}
&1/|a_{12}^-a_{12}^+|~~~&&{\rm if}~a_{12}^-\ne0,\\
&1 &&{\rm if}~a_{12}^-=0.
\end{aligned}
\right.
$$
Let
$$w:=\left\{
\begin{aligned}
\begin{array}{ll}
1 ~~~&{\rm if} ~A^2_{11}+4A_{21}(a_{12}^+)^2=0,\\
\left|A^2_{11}+4A_{21}(a_{12}^+)^2\right|^{1/2}/(2(a_{12}^+)^2)~~~&{\rm if}~A^2_{11}+4A_{21}(a_{12}^+)^2\ne0.
\end{array}
\end{aligned}
\right.$$
Applying the transformation $(x, y, t)\rightarrow(x/w, y, t/w)$ in (\ref{sflko}), we get
\begin{eqnarray}
\left\{
\begin{aligned}
\dot z&=\left(
\begin{array}{cc}
  2\alpha&1\\
  m-\alpha^2&0\\
\end{array}
\right)z+
\left(
\begin{array}{c}
  0\\
  \beta\\
\end{array}
\right)~~~~{\rm if}~x>0,\\
\dot z&=\left(
\begin{array}{cc}
  ~~\gamma_1~&~\delta~\\
  ~~\gamma_2~&~\gamma_3~\\
\end{array}
\right)z+
\left(
\begin{array}{c}
  \eta\\
  \rho\\
\end{array}
\right)~~~~~{\rm if}~x<0,
\end{aligned}
\right.
\label{dsgsh2}
\end{eqnarray}
where $\gamma_1=uC_{11}/w,~ \gamma_2=uC_{21}/w^2, ~\gamma_3=uC_{22}/w,~  \rho=uD_2/w$ and
\begin{eqnarray}
\begin{aligned}
&\alpha=A_{11}/(2w(a^+_{12})^2), ~~~~~~~&&\beta=B/(w(a_{12}^+)^2),\\
& \eta=uD_1, &&m={\rm sgn}\left\{A^2_{11}+4A_{21}(a_{12}^+)^2\right\}.
\end{aligned}
\label{abe}
\end{eqnarray}

Since $E_R$ of (\ref{DPWL}) is an admissible focus,
we get
\begin{eqnarray}
\begin{aligned}
&a_{11}^++a_{22}^+\ne 0,~~~a_{12}^+b_2^+-a_{22}^+b_1^+>0,\\
&(a_{11}^++a_{22}^+)^2+4(a_{12}^+a_{21}^+-a_{11}^+a_{22}^+)<0.
\end{aligned}
\label{a1221}
\end{eqnarray}
It follows from (\ref{ABD}), (\ref{abe}) and (\ref{a1221}) that
$\alpha\ne 0,~m=-1,~\beta>0,~\eta\ne0$ in system (\ref{dsgsh2}). That is, system (\ref{DPWL}) is transformed into
(\ref{dsgsh}). Observe that any used transformation is an orientation preserving diffeomorphism
and keeps the $y$-axis being the switching line. Eventually, by \cite[Proposition 2.22]{ChL-MG} system (\ref{DPWL}) is $\Sigma$-equivalent to
system (\ref{dsgsh}) under the given conditions.

It is easy to see that $a_{11}^++a_{22}^+\ne 0$ in (\ref{a1221}) is replaced by
 $a_{11}^++a_{22}^+>0$ (resp. $a_{11}^++a_{22}^+< 0$) if
the admissible focus $E_R$ of (\ref{DPWL}) is unstable (resp. stable).
So, we get the sign of $\alpha$ depending on the stability of this focus and the lemma is proved.
\end{proof}

\begin{figure}[htp]
  \begin{minipage}[t]{0.25\linewidth}
  \centering
  \includegraphics[width=1.4in]{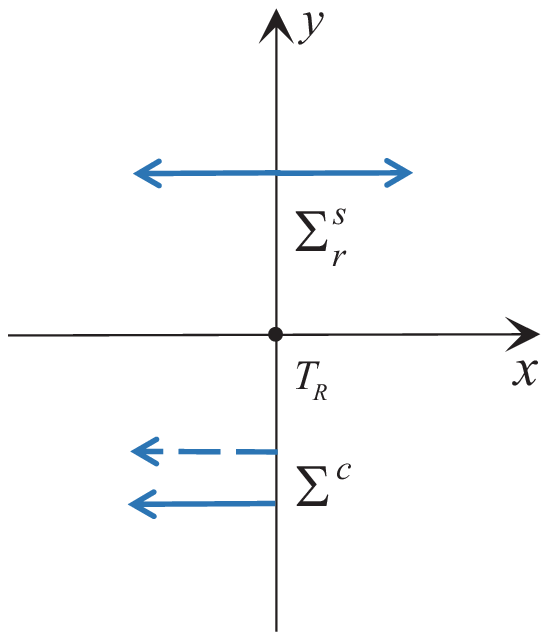}
  \caption*{(a) $\delta=0, \eta<0$}
  \end{minipage}
  \begin{minipage}[t]{0.25\linewidth}
  \centering
  \includegraphics[width=1.4in]{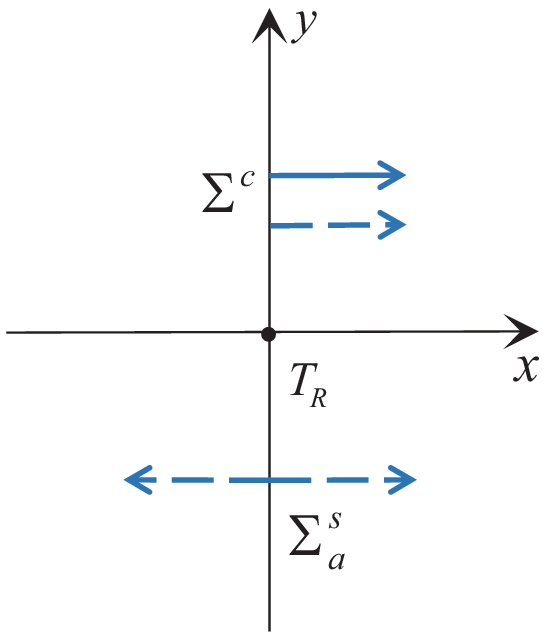}
  \caption*{(b) $\delta=0, \eta>0$}
  \end{minipage}
  \begin{minipage}[t]{0.25\linewidth}
  \centering
  \includegraphics[width=1.4in]{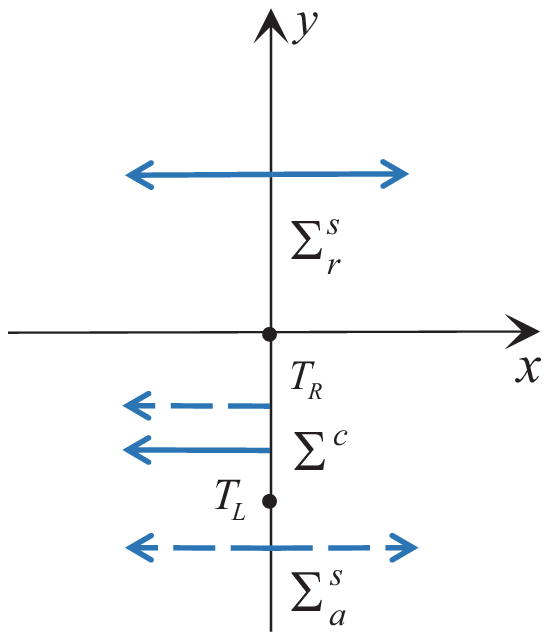}
  \caption*{(c) $\delta=-1, \eta<0$}
  \end{minipage}
  \begin{minipage}[t]{0.25\linewidth}
  \centering
  \includegraphics[width=1.4in]{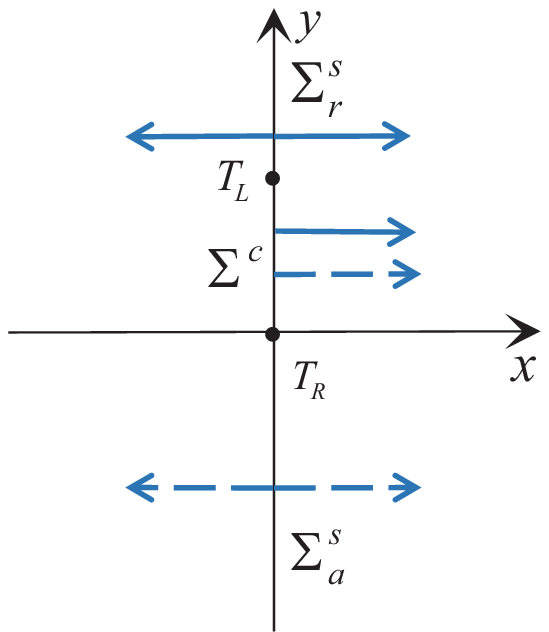}
  \caption*{(d) $\delta=-1, \eta>0$}
  \end{minipage}
  ~~~~~~~~~~~~~~~
  \begin{minipage}[t]{0.25\linewidth}
  \centering
  \includegraphics[width=1.4in]{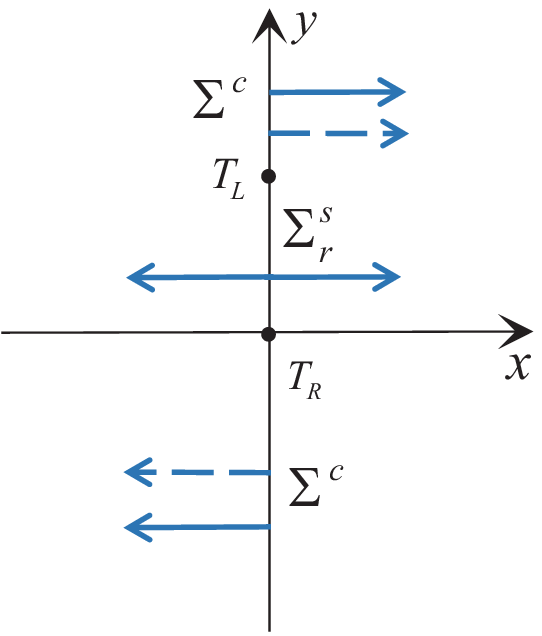}
  \caption*{(e) $\delta=1, \eta<0$}
  \end{minipage}
  ~~~~~~~~~~~~~~~
  \begin{minipage}[t]{0.25\linewidth}
  \centering
  \includegraphics[width=1.4in]{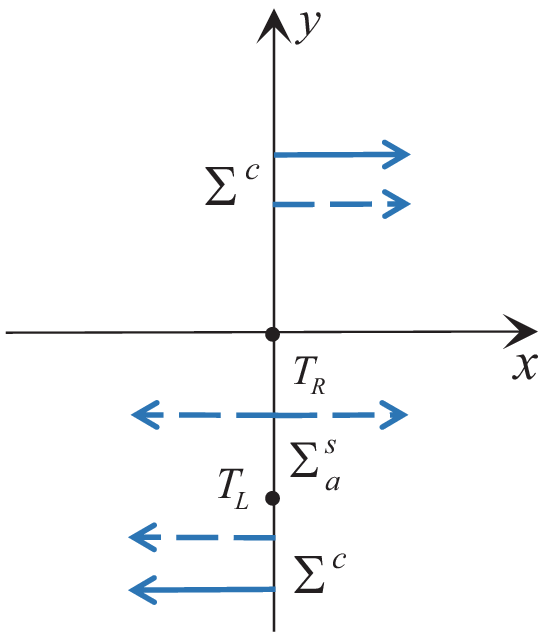}
  \caption*{(f) $\delta=1, \eta>0$}
  \end{minipage}
\caption{{\small Components of switching line $\Sigma$ of system (\ref{dsgsh})}}
\label{csl}
\end{figure}

By the definition of tangency point given in Section 1, it is easy to see that
the right system of (\ref{dsgsh}) has a unique tangency point $T_R:=(0, 0)^\top$ but
its left system may have zero or one tangency point depending on the signs of parameters $\delta$ and $\eta$.
In order to make the components of switching line $\Sigma$ clear for system (\ref{dsgsh}),
we give all possibilities in Figure \ref{csl} by the signs of $\delta$ and $\eta$.
We remark that in Figure \ref{csl} the arrows (resp. broken arrows) in the right hand plane denote the directions of
vector fields on $\Sigma$ of the right (resp. left) system, and
the arrows (resp. broken arrows) in the left hand plane denote the directions of
vector fields on $\Sigma$ of the left (resp. right) system.
Here $T_L$ denotes a tangency point of the left system.

Having these lemmas, we give the proof of Theorem~\ref{typeAnumber} in the following.

\begin{proof}[Proof of {\rm Theorem~\ref{typeAnumber}}]
For nondegenerate system (\ref{DPWL}), we first prove that the configuration of a sliding periodic orbit $\psi$ is one of Figure {\rm\ref{1A}}(a)-(d) in the sense of $\Sigma$-equivalence and time reversing. In fact, from the first paragraph in the proof of Lemma~\ref{suff}, $\psi$ cannot slide simultaneously on $\Sigma^s_a$ and $\Sigma^s_r$. Moreover, by an appropriate transformation of forms $(x, y, t)\rightarrow(\pm x, \pm y, \pm t)$, we always can assume that
$\psi$ slides on a segment of $\Sigma^s_a$ and the part in the right plane is as shown in Figure~\ref{part}.
That is, $\psi$ leaves $\Sigma^s_a$ at the visible tangency point $T_R$ of the right system of (\ref{DPWL}) and
reaches again $\Sigma$ at a point $T_0$ lying below $T_R$ from $\Sigma^+$ as $t$ increases.
\begin{figure}
  \begin{minipage}[t]{1.0\linewidth}
  \centering
  \includegraphics[width=1.50in]{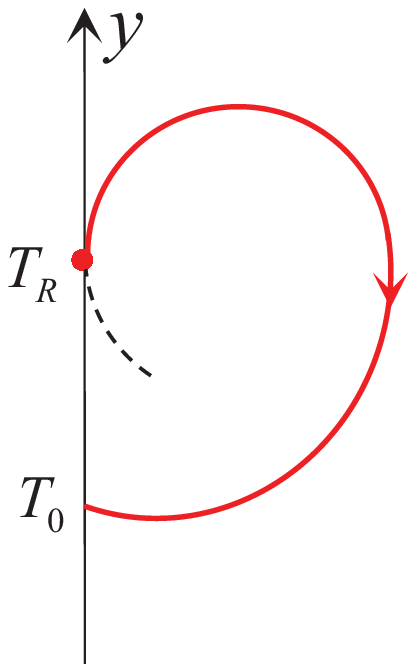}
  \end{minipage}
\caption{{\small A part of $\psi$ or $\tilde\psi$.}}
\label{part}
\end{figure}
Since (\ref{DPWL}) has a sliding periodic orbit, it can be $\Sigma$-equivalently
written as system (\ref{dsgsh}) by Lemmas \ref{suff} and \ref{norm1}. In system (\ref{dsgsh}), we
denote the sliding periodic orbit corresponding to $\psi$ by $\tilde\psi$.
We claim that the part of $\tilde\psi$ in the right plane is also as shown in Figure~\ref{part}.
In fact, from the forms of transformations in the proof of Lemma~\ref{norm1} we only need to check $a_{12}^+>0$
in (\ref{vch}). Clearly, the tangency point $T_R$ of the right system of (\ref{DPWL})
lies at $(0, -b_1^+/a_{12}^+)^\top$. Since the part of $\psi$ in the right plane is as shown in Figure~\ref{part},
equilibrium $E_R$ of (\ref{DPWL}) is an unstable admissible focus and $\dot y|_{T_R}=(-b_1^+/a_{12}^+)a_{22}^++b_2^+>0$.
Associated with the coordinate of $E_R$, we get $a_{12}^+b_2^+-a_{22}^+b_1^+>0$, implying $a_{12}^+>0$.
Thus, the part of $\tilde\psi$ in the right plane is also as shown in Figure~\ref{part}.

It is easy to check that the component of $\Sigma$ cannot be one of Figure~\ref{csl}(a)(c)(e) because
$\tilde\psi$ in the right plane is as shown in Figure~\ref{part}. So we only consider that the component of
$\Sigma$ is one of Figure~\ref{csl}(b)(d)(f). When the component of $\Sigma$ is one of Figure~\ref{csl}(b)(d),
we get $T_0\in\Sigma^s_a$. Thus $\tilde\psi$ slides from $T_0$ to $T_R$ along $\Sigma^s_a$ by the periodicity of $\tilde\psi$,
i.e., $\tilde\psi$ forms the configuration shown in Figure~\ref{1A}(a). When the component of $\Sigma$ is
Figure~\ref{csl}(f), either $T_0\in\Sigma^c$ or $T_0\in\Sigma^s_a$ or $T_0=T_L$. Thus $\tilde\psi$ either crosses
$\Sigma$ at $T_0$ or slides from $T_0$ along $\Sigma^s_a$. If $\tilde\psi$ crosses $\Sigma$ at $T_0$,
then $\tilde\psi$ is as shown in Figure~\ref{1A}(b) when $T_0\in\Sigma^c$ and as shown in Figure~\ref{1A}(c) when $T_0=T_L$.
If $\tilde\psi$ slides from $T_0$ along $\Sigma^s_a$, then it is as shown in Figure~\ref{1A}(a)
when the sliding direction is upward and as shown in either Figure~\ref{1A}(d) or Figure~\ref{Asingle}
when the sliding direction is downward. Therefore, the configuration of $\tilde\psi$ must be one of
Figure~\ref{1A}(a)-(d) and Figure~\ref{Asingle}. Because all transformations we used are time reversing and diffeomorphism keeping $y$-axis as switching line, the configuration of any sliding periodic orbit $\psi$ of (\ref{DPWL}) is one of Figure~\ref{1A}(a)-(d) and Figure~\ref{Asingle} in the sense of $\Sigma$-equivalence and time reversing. On the other hand, it is easy to observe that
Figure~\ref{Asingle} is equivalent to Figure~\ref{1A}(c) under the change $(x, y, t)\rightarrow(-x, -y, t)$.
Thus, the configuration of any sliding periodic orbit $\psi$ of (\ref{DPWL}) is one of Figure~\ref{1A}(a)-(d) in the sense of $\Sigma$-equivalence and time reversing.
\begin{figure}[htp]
\begin{minipage}[t]{1.0\linewidth}
  \centering
  \includegraphics[width=1.50in]{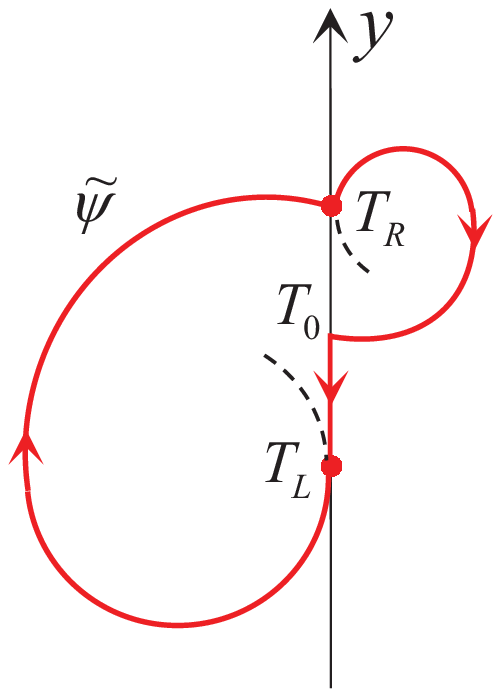}
  \end{minipage}
\caption{{\small Configuration of one sliding periodic orbit}}
\label{Asingle}
\end{figure}

From the above two paragraphs, we immediately obtain conclusion (i) if $\psi$ is a unique sliding periodic orbit of system (\ref{DPWL}).

Now we assume that system (\ref{DPWL}) has at least two sliding periodic orbits. Let
$\psi, \phi$ be two different ones among them. By the first two paragraphs of this proof, we assume that
$\psi, \phi$ of system (\ref{DPWL}) correspond to $\tilde\psi, \tilde\phi$ of system (\ref{dsgsh}) and
$\tilde\psi$ has configuration shown in one of Figure~\ref{1A}(a)-(d) and Figure~\ref{Asingle}.
As in the second paragraph of this proof, the component of $\Sigma$ is one of Figure~\ref{csl}(b)(d)(f).
It is easy to observe that there exists a unique sliding periodic orbit if the component of $\Sigma$
is of form Figure~\ref{csl}(b). Thus the component of $\Sigma$ is one of Figure~\ref{csl}(d)(f).

\begin{figure}[htp]
  \begin{minipage}[t]{1.0\linewidth}
  \centering
  \includegraphics[width=1.55in]{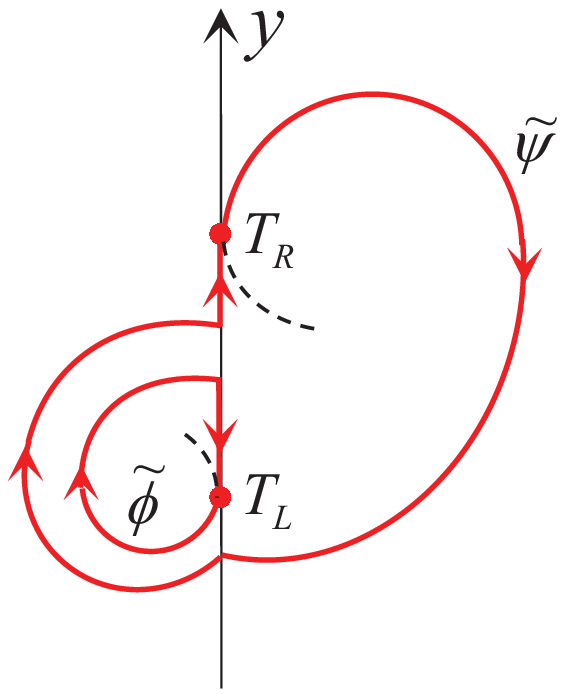}
  \end{minipage}
\caption{{\small Configuration of two sliding periodic orbits}}
\label{f}
\end{figure}

If the component of $\Sigma$ is of Figure~\ref{csl}(d), we find that $\tilde\psi$ must be the configuration of Figure~\ref{1A}(a)
and $\tilde\phi$ slides on $\Sigma^s_r$. Thus we get the configuration of $\tilde\psi$
and $\tilde\phi$ as shown in Figure~\ref{2A}(a). If the component of $\Sigma$ is of Figure~\ref{csl}(f),
as in the second paragraph of this proof, $\tilde\psi$ is of either one of
Figure~\ref{1A}(a)-(d) or Figure~\ref{Asingle}. We get the configuration of $\tilde\psi$ and $\tilde\phi$ as shown in
Figure~\ref{2A}(b)(c) when $\tilde\psi$ is of Figure~\ref{1A}(a), as shown in
Figure~\ref{f} when $\tilde\psi$ is of Figure~\ref{1A}(b).
However, $\tilde\psi$ is a unique sliding periodic orbit when
it is either Figure~\ref{1A}(c) or Figure~\ref{1A}(d) or Figure~\ref{Asingle}.
Thus all possible configurations of $\tilde\psi$ and $\tilde\phi$ of system (\ref{dsgsh}) are as shown in Figure~\ref{2A}(a)-(c) and Figure~\ref{f}.
This means that the configuration of any two sliding periodic orbits $\psi$ and $\phi$ of system (\ref{DPWL})
is one of Figure~\ref{2A}(a)-(c) and Figure~\ref{f} in the sense of $\Sigma$-equivalence and time reversing.
On the other hand, it is easy to observe that Figure~\ref{f} is equivalent to Figure~\ref{2A}(c) under the change $(x, y, t)\rightarrow(-x, -y, t)$.
Therefore, the configuration of $\psi$ and $\phi$ of (\ref{DPWL}) is one of Figure~\ref{2A}(a)-(c) in the sense of $\Sigma$-equivalence and time reversing, i.e., conclusion (ii) holds.

From the definition of sliding periodic orbits, any such orbit must pass through
$T_L$ or $T_R$. Associated with the configurations of $\psi$ and $\phi$ obtained in the above,
we get that $2$ is the maximum number of sliding periodic orbits and the proof is completed.
\end{proof}

\section{Proof of Theorem~\ref{AC}}
\setcounter{equation}{0}
\setcounter{lm}{0}
\setcounter{thm}{0}
\setcounter{rmk}{0}
\setcounter{df}{0}
\setcounter{cor}{0}

The purpose of this section is to give the proof of Theorem~\ref{AC}. For brevity, we define $\tau:=\gamma_1+\gamma_3$ and $\Delta:=\gamma_1\gamma_3-\gamma_2$ in the following. To this end, we need some preliminaries. Assume that the eight parameters $\alpha, \beta, \delta, \eta, \rho, \gamma_1, \gamma_2, \gamma_3$
in system (\ref{dsgsh}) satisfy
\begin{eqnarray}
\begin{aligned}
&\alpha>0, &&\beta>0, &&\eta>0,&&\rho-\gamma_3\eta<0, \\
&\delta=1, ~~~~~&&\gamma_1=\gamma_3,~~~~&&\tau>0, ~~~~~&&\tau^2<4\Delta.~~&&
\end{aligned}
\label{addcond}
\end{eqnarray}
It is easy to check that for the left system in (\ref{dsgsh}) the unique equilibrium lies at
$$(\overline x, \overline y)^\top:=\left(\frac{\rho-\gamma_3\eta}{\Delta}, \frac{\eta\gamma_2-\rho\gamma_1}{\Delta}\right)^\top$$
and is an unstable focus, and $(0,-\eta)^\top$ is the tangency point.
All orbits in a small neighborhood of this equilibrium rotate clockwise.
Note that $\overline x<0$ because $\rho-\gamma_3\eta<0$ and
$\Delta>0$ by (\ref{addcond}).
By straight computations, we get
the solution satisfying $(x(0), y(0))^\top=(0, y_0)^\top$ of the left system
\begin{eqnarray}
\left(
\begin{aligned}
x(t)\\
y(t)
\end{aligned}
\right)=e^{\gamma_3t}
\left(
\begin{aligned}
  -\overline x{\rm cos}\nu t+\frac{1}{\nu }(y_0-\overline y){\rm sin}\nu t\\
  -\frac{\gamma_2\overline x}{\nu }{\rm sin}\nu t+(y_0-\overline y){\rm cos}\nu t\\
\end{aligned}
\right)+\left(
\begin{aligned}
\overline x\\
\overline y
\end{aligned}
\right),
\label{solutions1}
\end{eqnarray}
where $\nu :=\sqrt{|\gamma_2|}$.
Let $t^-$ be the minimum time for the orbit of (\ref{dsgsh}) from $(0,y_0)^\top$
to intersect $y$-axis, where $y_0\le -\eta$.
Clearly, $\nu t^-\in (\pi, 2\pi]$ because of the linearity of (\ref{dsgsh}).
By (\ref{solutions1}) we get $\sin \nu t^-\ne 0$ and
\begin{eqnarray}
\begin{aligned}
y_0&=\overline y+\frac{\left(\cos \nu t^--e^{-\gamma_3t^-}\right)\nu}{\sin\nu t^-}\overline x
=-\eta-\frac{\nu (\rho-\gamma_3\eta)}{\Delta}\frac{\varphi_+(t^-)}{{\rm sin}\nu t^-}e^{-\gamma_3t^-},\\
y(t^-)&=e^{\gamma_3t^-}\left(-\frac{\gamma_2\overline x}{\nu}\sin\nu t^- +(y_0-\overline y)\cos \nu t^-\right)+\overline y\\
&=-\eta+\frac{\nu (\rho-\gamma_3\eta)}{\Delta}\frac{\varphi_-(t^-)}{{\rm sin}\nu t^-}e^{\gamma_3t^-},
\end{aligned}
\label{lpoincare}
\end{eqnarray}
where
$$\varphi_\pm(t^-)=1-e^{\pm\gamma_3t^-}({\rm cos}\nu t^-\mp\frac{\gamma_3}{\nu }{\rm sin}\nu t^-).$$
Obviously, $t^-=\hat t^-$ when $y_0=-\eta$, where
$\hat t^-\in (\pi/\nu, 2\pi/\nu]$ and satisfies $\varphi_+(\hat t^-)=0$ by the first equality in (\ref{lpoincare}).
It is easy to prove the uniqueness of $\hat t^-$ by the expression of $\varphi_+(t^-)$.
Thus, for $y_0\le -\eta$ we get $t^-\in (\pi/\nu, \hat t^-]$.
Define a left Poincar\'e map
\begin{eqnarray*}
P_L:~\{y\in\mathbb{R}:~y\le-\eta\}&\to& \{y\in\mathbb{R}:~y\ge y_\eta\}\\
~~~~~~~~~~y_0&\mapsto& y(t^-),
\end{eqnarray*}
where $y_\eta$ is the $y$-coordinate of the first intersection point between $y$-axis and the orbit
of (\ref{dsgsh}) from $(0,-\eta)^\top$. Using (\ref{lpoincare}), we have the reverse $P_L^{-1}$ of $P_L$ for $y\ge y_\eta$
in parametric form
\begin{eqnarray}
\begin{aligned}
y&=-\eta+\frac{\nu (\rho-\gamma_3\eta)}{\Delta}\frac{\varphi_-(t^-)}{{\rm sin}\nu t^-}e^{\gamma_3t^-},\\
P^{-1}_L(y)&=-\eta-\frac{\nu (\rho-\gamma_3\eta)}{\Delta}\frac{\varphi_+(t^-)}{{\rm sin}\nu t^-}e^{-\gamma_3t^-},
\end{aligned}
\label{leftp}
\end{eqnarray}
where $t^-\in (\pi/\nu, \hat t^-]$.

For the right system in (\ref{dsgsh}),
it is easy to check that the unique equilibrium
is an unstable admissible focus and $(0, 0)^\top$ is the tangency point.
All orbits in a small neighborhood of this equilibrium rotate clockwise.
Similarly to the left system, for the right system we can define the right Poincar\'e map $P_R$ and express it
in the parametric form as
\begin{eqnarray}
y=-\frac{\beta}{1+\alpha^2}\frac{e^{-\alpha t^+}\psi_+(t^+)}{{\rm sin}t^+},~~~~
P_R(y)=\frac{\beta}{1+\alpha^2}\frac{e^{\alpha t^+}\psi_-(t^+)}{{\rm sin}t^+}
\label{rightp}
\end{eqnarray}
for $y\ge 0$,
where $t^+\in (\pi, \hat t^+]$ and
$$
\psi_{\pm}(t^+):=1-e^{\pm\alpha t^+}\left(\cos t^+\mp\alpha\sin t^+\right).
$$
Here $\hat t^+\in (\pi, 2\pi]$ is unique and satisfies $\psi_+(\hat t^+)=0$.

In the following lemma, we give some properties of $P^{-1}_L(y)$ and $P_R(y)$.

\begin{lm}
Assume that condition {\rm(\ref{addcond})} holds.\\
{\rm (i)} For $y>y_\eta$, we have $P^{-1}_L(y)<-\eta$ and
\begin{eqnarray}
\lim_{y\rightarrow+\infty}\frac{dP^{-1}_L(y)}{dy}=-e^{-\frac{\gamma_3}{\nu}\pi},
~~~~~~~~~\frac{d^2P^{-1}_L(y)}{dy^2}>0.
\label{lpprop}
\end{eqnarray}
{\rm (ii)} For $y>0$, we have $P_R(y)<0$ and
\begin{eqnarray}
\lim_{y\rightarrow+\infty}\frac{dP_R(y)}{dy}=-e^{\alpha\pi},
~~~~~~~~~~~~~~~~~~~\frac{d^2P_R(y)}{dy^2}<0.
\label{rpprop}
\end{eqnarray}
\label{lp}
\end{lm}

Notice that the above partial results can also be obtained in a different canonical form of system (\ref{DPWL}) (see \cite{ChL-EE, ChL-FEEE}).
However, for convenience and completeness we still present it here in our canonical form and give a short proof.

\begin{proof}
The first part of conclusion (i) follows directly the definition of $P^{-1}_L$ given in (\ref{leftp}). By the parametric form of $P^{-1}_L$
given in (\ref{leftp}) we get
\begin{eqnarray}
\frac{dy}{dt^-}=-\frac{\nu\left(P^{-1}_L(y)+\eta\right)}{\sin \nu t^-}e^{\gamma_3t^-},~~~~~~~
\frac{dP^{-1}_L(y)}{dt^-}=-\frac{\nu\left(y+\eta\right)}{\sin \nu t^-}e^{-\gamma_3t^-}.
\label{dert}
\end{eqnarray}
Then
\begin{eqnarray}
\frac{dP^{-1}_L(y)}{dy}=\frac{y+\eta}{P^{-1}_L(y)+\eta}e^{-2\gamma_3 t^-}.
\label{dery}
\end{eqnarray}
Since the orbits of the left system in (\ref{dsgsh}) rotate at a steady speed surrounding the left equilibrium,
$t^-\rightarrow\pi/\nu$ as $y\rightarrow+\infty$. Thus, by (\ref{dery}) and the decreasing of $P^{-1}_L(y)$
we have the first equality in (\ref{lpprop}).
Further, using (\ref{dery}) we obtain
{\small\begin{eqnarray}
\frac{d^2P^{-1}_L(y)}{dy^2}(P^{-1}_L(y)+\eta)\!+\!\left(\frac{dP^{-1}_L(y)}{dy}\right)^2\!=\!e^{-2\gamma_3 t^-}
\!-\!\frac{2\gamma_3(y+\eta)e^{-2\gamma_3 t^-}}{\frac{dy}{~dt^-}}.
\label{ddjah}
\end{eqnarray}}
By (\ref{leftp}), (\ref{dert}) and (\ref{ddjah}) we get
$$
\begin{aligned}
\frac{d^2P^{-1}_L(y)}{dy^2}&=\left(1-2\gamma_3
(y+\eta)\frac{dt^-}{dy}-\left(\frac{y+\eta}{P^{-1}_L(y)+\eta}\right)^2e^{-2\gamma_3 t^-}\right)
\frac{e^{-2\gamma_3t^-}}{P^{-1}_L(y)+\eta}\\
&=\left((P^{-1}_L(y)+\eta)^2\left(1-\frac{2\gamma_3
(y+\eta)}{dy/dt^-}\right)-(y+\eta)^2e^{-2\gamma_3 t^-}\right)
\frac{e^{-2\gamma_3 t^-}}{(P^{-1}_L(y)+\eta)^3}\\
&=\left(\!\varphi^2_+(t^-)\!\left(\!1\!-\frac{2\gamma_3\varphi_-(t^-){\rm sin}\nu t^-}{\nu e^{-\gamma_3t^-}\varphi_+(t^-)}\!\right)\!-\varphi^2_-(t^-)e^
{2\gamma_3t^-}\!\right)\!\frac{\nu ^2\overline x^2}{({\rm sin}\nu t^-)^2}
\frac{e^{-4\gamma_3 t^-}}{(P^{-1}_L(y)\!+\!\eta)^3}\\
&=-2\overline x^2  \Delta\left({\rm sinh}\gamma_3t^--\frac{\gamma_3}{\nu }{\rm sin}\nu t^-\right)
\frac{e^{-3\gamma_3t^-}}{(P^{-1}_L(y)+\eta)^3}\\
&=-\frac{2(\rho-\gamma_3\eta)^2}{\Delta}\frac{{\rm sinh}\gamma_3t^--\frac{\gamma_3}{\nu }{\rm sin}\nu t^-}{(P^{-1}_L(y)+\eta)^3}e^{-3\gamma_3t^-}.
\end{aligned}
$$
As indicated in \cite{ChL-EE},
$${\rm sign}\left({\rm sinh}\gamma_3t^--\frac{\gamma_3}{\nu }
{\rm sin}\nu t^-\right)={\rm sign}\gamma_3>0.$$
Note that $\Delta>0$ as shown below (\ref{addcond}).
This means
$$\frac{d^2P^{-1}_L(y)}{dy^2}>0.$$
The proof of conclusion (i) is finished.

The first part of conclusion (ii) is obvious because of the definition of $P_R$ given in (\ref{rightp}).
In order to prove the second part of (ii), similarly to $P^{-1}_L$ we firstly get
\begin{eqnarray}
\frac{dP_R(y)}{dy}=\frac{y}{P_R(y)}e^{2\alpha t^+}, ~~~~~
\frac{d^2P_R(y)}{dy^2}=\frac{2\beta^2}{1+\alpha^2}\frac{{\rm sinh}\alpha t^+-\alpha{\rm sin}t^+}{(P_R(y))^3}e^{3\alpha t^+}.
\label{dery1}
\end{eqnarray}
Then we get (\ref{rpprop}) for $P_R(y)$ as we do for $P^{-1}_L(y)$ in last paragraph.
\end{proof}

Having these preliminaries, we give the proof of Theorem \ref{AC}.

\begin{proof}[Proof of {\rm Theorem~\ref{AC}}]
If system (\ref{DPWL}) has a sliding periodic orbit, then it is $\Sigma$-equivalent to (\ref{dsgsh}) by Lemmas~\ref{suff}, \ref{norm1} and the definition
of $\Sigma$-equivalence given below the Theorem~\ref{typeAnumber}. Thus we only need to prove Theorem~\ref{AC} for system (\ref{dsgsh}).

If system (\ref{dsgsh}) has two sliding periodic orbits shown in Figure~\ref{2A}(a), switching line $\Sigma$ of (\ref{dsgsh}) is as Figure~\ref{csl}(d).
From Figure~\ref{csl}(d) we observe that the direction of the vector field on $\Sigma^c$ is always rightward.
Thus (\ref{dsgsh}) has no crossing periodic orbits. The conclusion (i) of this theorem is proved.

If system (\ref{dsgsh}) has two sliding periodic orbits shown in one of Figure~\ref{2A}(b)(c), switching line $\Sigma$ of (\ref{dsgsh}) is as Figure~\ref{csl}(f),
implying
\begin{eqnarray}
\delta=1,~~~~\eta>0.
\label{awkf}
\end{eqnarray}
Moreover, we observe that both equilibria $E_L$ and $E_R$ are unstable admissible foci in Figure~\ref{2A}(b)(c). Thus
\begin{eqnarray}
\alpha>0,~~~\beta>0, ~~~\tau>0, ~~~\tau^2-4\Delta<0, ~~~\rho-\gamma_3\eta<0
\label{eiosfjn}
\end{eqnarray}
in system (\ref{dsgsh}), where $\tau=\gamma_1+\gamma_3$ and $\Delta=\gamma_1\gamma_3-\gamma_2$.
We claim that system (\ref{dsgsh}) with (\ref{awkf}) and (\ref{eiosfjn}) has exactly one crossing periodic orbit.
In fact, by the continuous transformation
$$
z\rightarrow\left\{
\begin{aligned}
&\left(\begin{array}{cc}
1 &0\\
0 &1
\end{array}\right)z~~~~~~&&{\rm if}~~x\ge0,\\
&
\left(\begin{array}{cc}
1 &0\\
\kappa &1
\end{array}\right)z~~&&{\rm if}~~x<0
\end{aligned}
\right.
$$
with $\kappa=(\gamma_3-\gamma_1)/2$, we can always assume that system (\ref{dsgsh}) satisfies $\gamma_1=\gamma_3$, (\ref{awkf}) and (\ref{eiosfjn}), i.e., condition (\ref{addcond}). Therefore, we can equivalently prove that system (\ref{dsgsh}) with (\ref{addcond}) has exactly one crossing periodic orbit.
Let $D(y):=P_L^{-1}(y)-P_R(y)$ for all $y\ge y^*:=\max\{y_\eta, 0\}$.
Then the number of crossing periodic orbits of (\ref{dsgsh}) equals to the number of
zeros of $D(y)$. Observing Figure~\ref{2A}(b)(c), we get $D(y^*)<0$. On the other hand, by Lemma~\ref{lp} we get
$$\lim_{y\rightarrow+\infty}D'(y)=\lim_{y\rightarrow+\infty}\frac{dP^{-1}_L(y)}{dy}-\lim_{y\rightarrow+\infty}\frac{dP_R(y)}{dy}
=e^{\alpha\pi}-e^{-\frac{\gamma_3}{\nu}\pi}>0$$
and $D''(y)>0$. The former implies $D(y)\rightarrow+\infty(y\rightarrow+\infty)$, from which we obtain the existence of
an unstable crossing periodic orbit, while the latter implies that system (\ref{dsgsh}) with (\ref{addcond}) has exactly one crossing periodic orbit.
Finally, if it has two sliding periodic orbits shown in one of Figure~\ref{2A}(b)(c), system (\ref{dsgsh}) has exactly one crossing periodic orbit, which is unstable.

If system (\ref{dsgsh}) has a unique sliding periodic orbit shown in one of Figure~\ref{1A}(c)(d), the conditions (\ref{awkf}) and (\ref{eiosfjn}) still hold. Moreover, we can similarly consider system (\ref{dsgsh}) with (\ref{addcond}) and prove that $D(y^*)<0$, $D(y)\rightarrow+\infty(y\rightarrow+\infty)$ and $D''(y)>0$.
Thus (\ref{dsgsh}) has also exactly one crossing periodic orbit, which is unstable. Associate with the last paragraph, the conclusion (ii) of this theorem is proved.

To prove conclusion (iii), let us consider $\alpha>0, \beta=\delta=\eta=1, \gamma_1=-\gamma_2=2, \gamma_3=0, \rho<0$ in system (\ref{dsgsh}), i.e.,
\begin{eqnarray}
\left\{
\begin{aligned}
\dot z&=\left(
\begin{array}{cc}
  2\alpha&1\\
  -1-\alpha^2&0\\
\end{array}
\right)z+
\left(
\begin{array}{c}
  0\\
  1\\
\end{array}
\right)~~~~&&{\rm if}~x>0,\\
\dot z&=\left(
\begin{array}{cc}
  ~2~&~~1~\\
  ~-2~&~~0~\\
\end{array}
\right)z+
\left(
\begin{array}{c}
  1\\
  \rho\\
\end{array}
\right)~~~~~&&{\rm if}~x<0.
\end{aligned}
\right.
\label{example1}
\end{eqnarray}
We claim that there exist a sufficiently small $\epsilon_1>0$ and a function $\rho_c(\alpha)<0$ such that
system (\ref{example1}) has two crossing periodic orbits $\Psi^0, \Gamma_{CC}$ and one
sliding periodic orbit $\Phi$ which is as Figure~\ref{1A}(a) if $0<\alpha<\epsilon_1$ and $\rho=\rho_c(\alpha)$
(see Figure~\ref{EX-1}(b)).
In fact, we have tangency points $T_L=(0, -1)^\top, T_R=(0, 0)^\top$ and switching line $\Sigma$ is as Figure~\ref{csl}(f). Thus, by
(\ref{sliding}) and (\ref{sliding0}) the sliding vector field is $(0, (1-\rho)y+1)^\top$ on $\Sigma^s_a=\{(0, y)^\top: -1<y<0\}$.
This implies that (\ref{example1}) has a unique pseudo-equilibrium at $(0, 1/(\rho-1))^\top$ and
the direction of the sliding vector field is upward (resp. downward) for $y>1/(\rho-1)$ (resp.$y<1/(\rho-1)$).

\begin{figure}[thp]
  \begin{minipage}[t]{0.25\linewidth}
  \centering
  \includegraphics[width=1.4in]{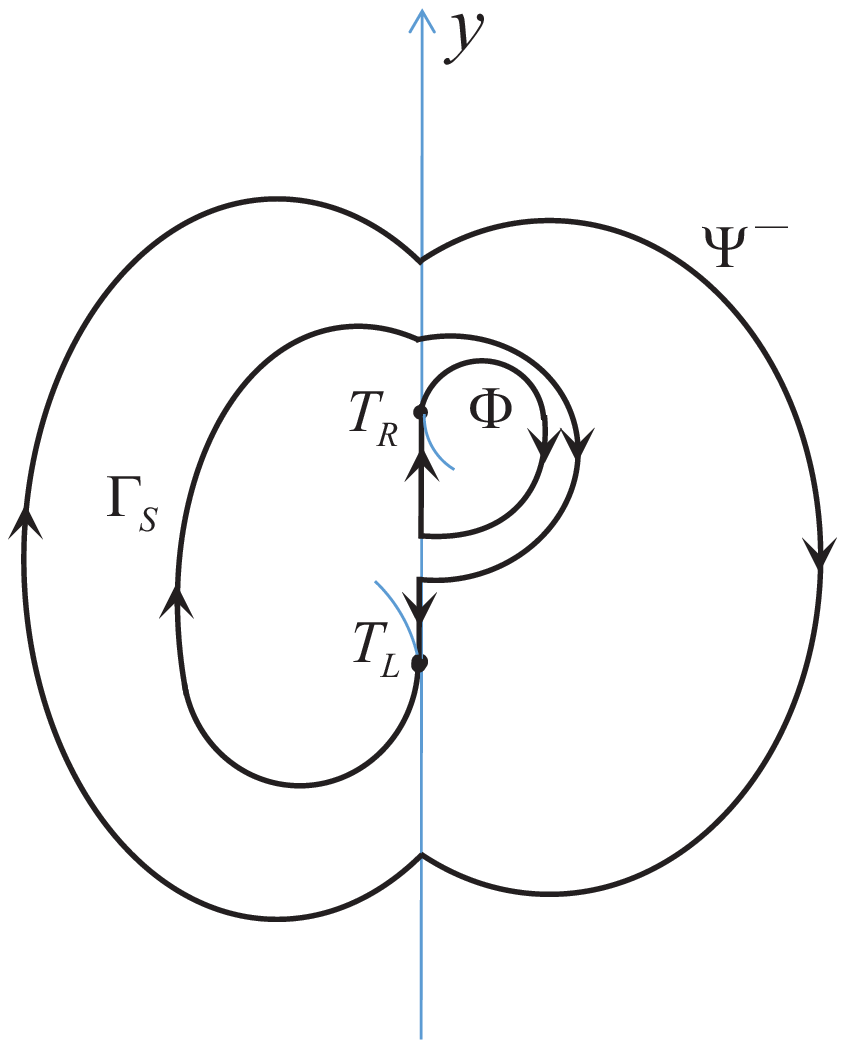}
  \caption*{(a) {\small $\rho_c(\alpha)\!<\!\rho\!<\!\rho_c(\alpha)\!+\!\epsilon_2$}}
  \end{minipage}
  ~~~~~~~~~~~
  \begin{minipage}[t]{0.25\linewidth}
  \centering
  \includegraphics[width=1.4in]{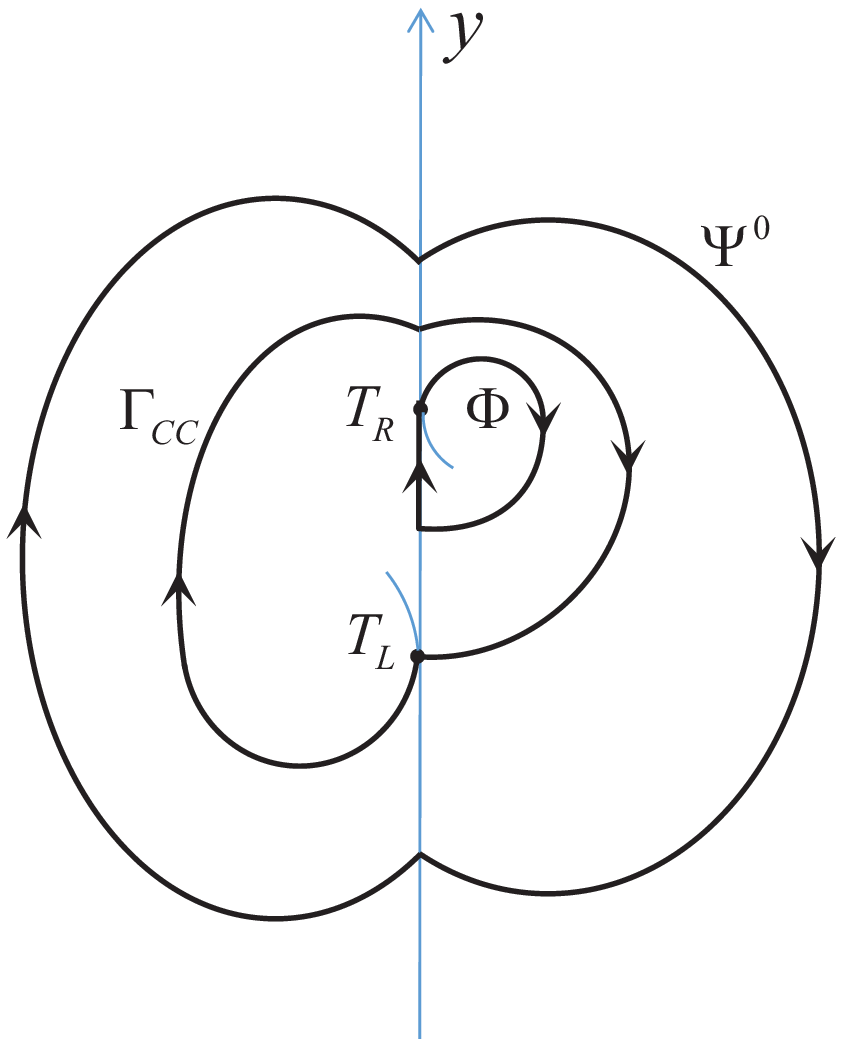}
  \caption*{(b) {\small $\rho=\rho_c(\alpha)$}}
  \end{minipage}~~~~~~~~~~~
  \begin{minipage}[t]{0.25\linewidth}
  \centering
  \includegraphics[width=1.37in]{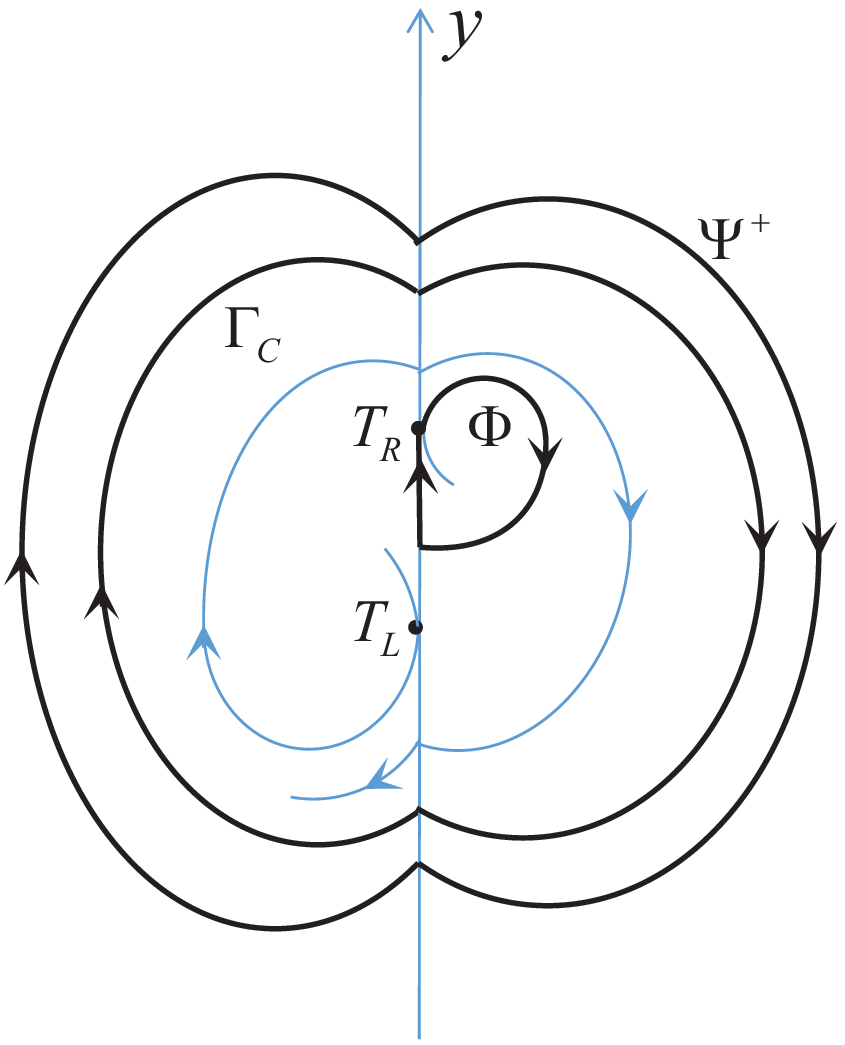}
  \caption*{(c) {\small $\rho_c(\alpha)\!-\!\epsilon_2\!<\!\rho\!<\!\rho_c(\alpha)$}}
  \end{minipage}
\caption{{\small Phase portrait of system (\ref{example1}) with $0<\alpha<\epsilon_1$}}
\label{EX-1}
\end{figure}

On one hand, since $E_R=(1/(\alpha^2+1), -2\alpha/(\alpha^2+1))^\top$ is an unstable admissible focus, $T_R$ is visible and
the orbit of the right system starting from $T_R$ will reach again $\Sigma$ from the right half plane. Denote the reaching point
by $(0, y_1(\alpha))^\top$, by (\ref{rightp}) we have $y_1(\alpha)=e^{\alpha \hat t^+}\sin \hat t^+$.
Clearly, $y_1(\alpha)\rightarrow 0^-$ as $\alpha\rightarrow 0^+$ because $\hat t^+\to 2\pi$ as $\alpha\rightarrow 0^+$.
When $\alpha$ is sufficiently small, $y_1(\alpha)>-1$, which implies that
there exists a point $(0, y_2(\alpha))^\top$ such that the orbit of the right system
starting from it will reach again $\Sigma$ at $T_L$. By the form of the right system,
$y_2(\alpha)\rightarrow 1$ as $\alpha\rightarrow 0^+$. On the other hand,
since $E_L=(\rho/2, -\rho-1)^\top$ is also an unstable admissible focus, $T_L$ is visible
and the orbit of the left system starting from $T_L$ will reach again $\Sigma$ from the left half plane.
Denote the reaching point by $(0, y_3(\rho))^\top$. By (\ref{lpoincare}) we get
$y_3(\rho)=-1+\rho\sin\hat t^-e^{\hat t^-}$, where $\hat t^-\in(\pi, 2\pi)$ satisfies $\varphi_+(\hat t^-)=0$ and is a constant
independent of $\rho$. Clearly, $y_3(\hat \rho)=1$, where $\hat \rho:=2/(\sin\hat t^-e^{\hat t^-})$, and
the unique pseudo-equilibrium lies at $(0, 1/(\hat \rho-1))^\top$ when $\rho=\hat \rho$.
By the continuities of $y_1(\alpha), y_2(\alpha), y_3(\rho)$, there exist a sufficiently small $\epsilon_1>0$ and a constant $\rho_c(\alpha)$
near $\hat \rho$ such that $y_3(\rho_c(\alpha))=y_2(\alpha)$, $y_1(\alpha)$ lies in a sufficiently small neighborhood of $0$
and the unique pseudo-equilibrium lies in a sufficiently small neighborhood of $(0, 1/(\hat \rho-1))^\top$
for all $\alpha\in (0, \epsilon_1)$.  That is, we have two periodic orbits $\Phi$ and $\Gamma_{CC}$ as shown in
Figure~\ref{EX-1}(b).

Now we prove the existence of periodic orbit $\Psi^0$ as shown in Figure~\ref{EX-1}(b) when $\rho=\rho_c(\alpha)$.
It is easy to check that (\ref{example1}) satisfies conditions (\ref{awkf}) and (\ref{eiosfjn}).
A similar analysis to the third paragraph shows that
$D(y)\rightarrow+\infty$($y\rightarrow+\infty$) and $D''(y)>0$. By a straight computation, we get $D(y_3(\rho_c(\alpha)))=0$.
When $y\rightarrow y_3(\rho_c(\alpha))$, $dP^{-1}_L(y)/dy$ and $dP_R(y)/dy$ converge to $-\infty$ and a finite constant
by (\ref{dery}) and (\ref{dery1}). So $D'(y)\rightarrow-\infty$ as $y\rightarrow y_3(\rho_c(\alpha))$,
which implies $D(y)<0$ for $y>y_3(\rho_c(\alpha))$ sufficiently close to $y_3(\rho_c(\alpha))$.
By Zero Point Theorem (\ref{example1}) has a crossing periodic orbit $\Psi^0$ which is different from $\Gamma_{CC}$
if $0<\alpha<\epsilon_1$ and $\rho=\rho_c(\alpha)$. According to the last paragraph, the claim is proved and then conclusion (iii) holds.

To prove conclusion (iv), let us consider
$\alpha=\beta=\delta=1, \gamma_1<-2, \gamma_2=-1-\gamma_1^2/4, \gamma_3=0, \eta>0,
\rho=(4+\gamma_1^2)(e^{2\pi}-1)/8$ in system (\ref{dsgsh}), i.e.,
\begin{eqnarray}
\left\{
\begin{aligned}
\dot z&=\left(
\begin{array}{cc}
  ~2~&~~1~~\\
  ~-2~&~~0~~\\
\end{array}
\right)z+
\left(
\begin{array}{c}
  0\\
  1\\
\end{array}
\right)~~~~&&{\rm if}~x>0,\\
\dot z&=\left(
\begin{array}{cc}
  \gamma_1&1\\
  -1-\frac{\gamma_1^2}{4}&0\\
\end{array}
\right)z+
\left(
\begin{array}{c}
  \eta\\
  \frac{(4+\gamma_1^2)(e^{2\pi}-1)}{8}\\
\end{array}
\right)~~~~~&&{\rm if}~x<0.
\end{aligned}
\right.
\label{example2}
\end{eqnarray}
We claim that there exist two sufficiently small constants $\epsilon_3>0, \epsilon_4>0$ and a function $\eta_c(\gamma_1)$
such that system (\ref{example2}) has two crossing periodic orbits $\Phi^+, \Psi^+$ and
one sliding periodic orbit $\Gamma_S$ which is as Figure~\ref{1A}(b)
if $-2-\epsilon_3<\gamma_1<-2$ and $\eta_c(\gamma_1)<\eta<\eta_c(\gamma_1)+\epsilon_4$ (see Figure~\ref{EX-2}(c)).

\begin{figure}[thp]
  \begin{minipage}[t]{0.25\linewidth}
  \centering
  \includegraphics[width=1.37in]{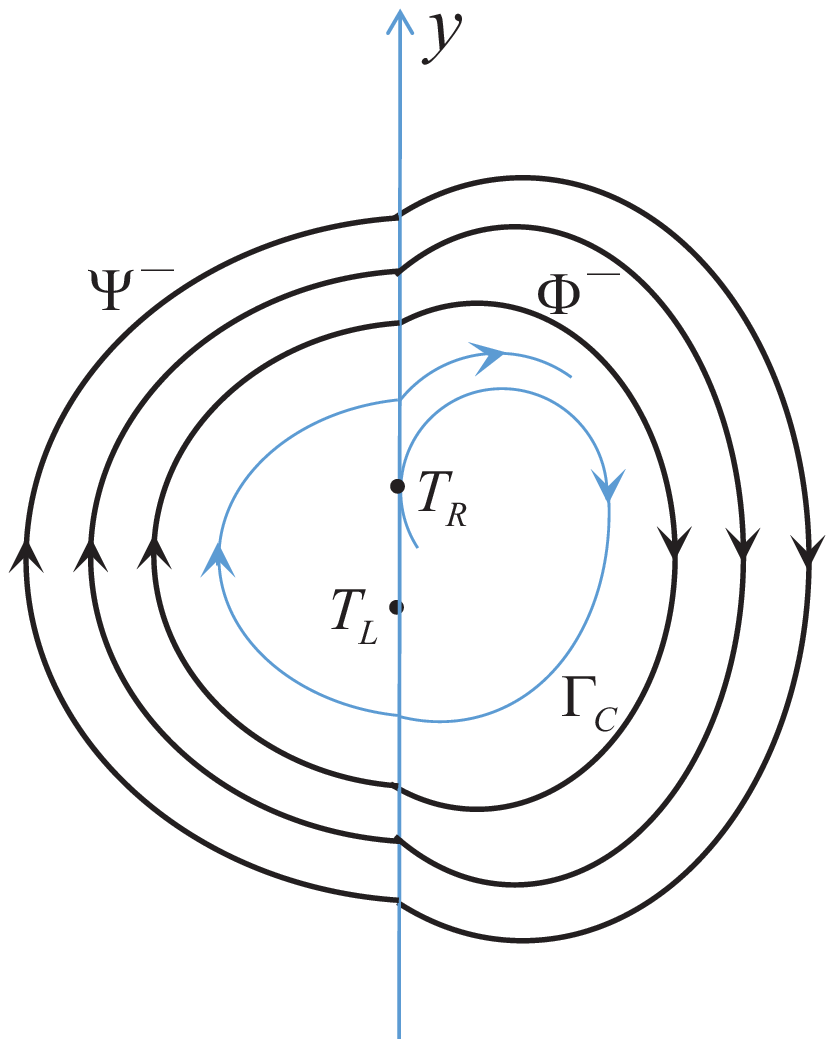}
  \caption*{(a) {\small $\eta_c(\gamma_1)\!-\!\epsilon_4\!\!<\!\eta\!<\!\eta_c(\gamma_1)$}}
  \end{minipage}~~~~~~~~~~~
  \begin{minipage}[t]{0.25\linewidth}
  \centering
  \includegraphics[width=1.30in]{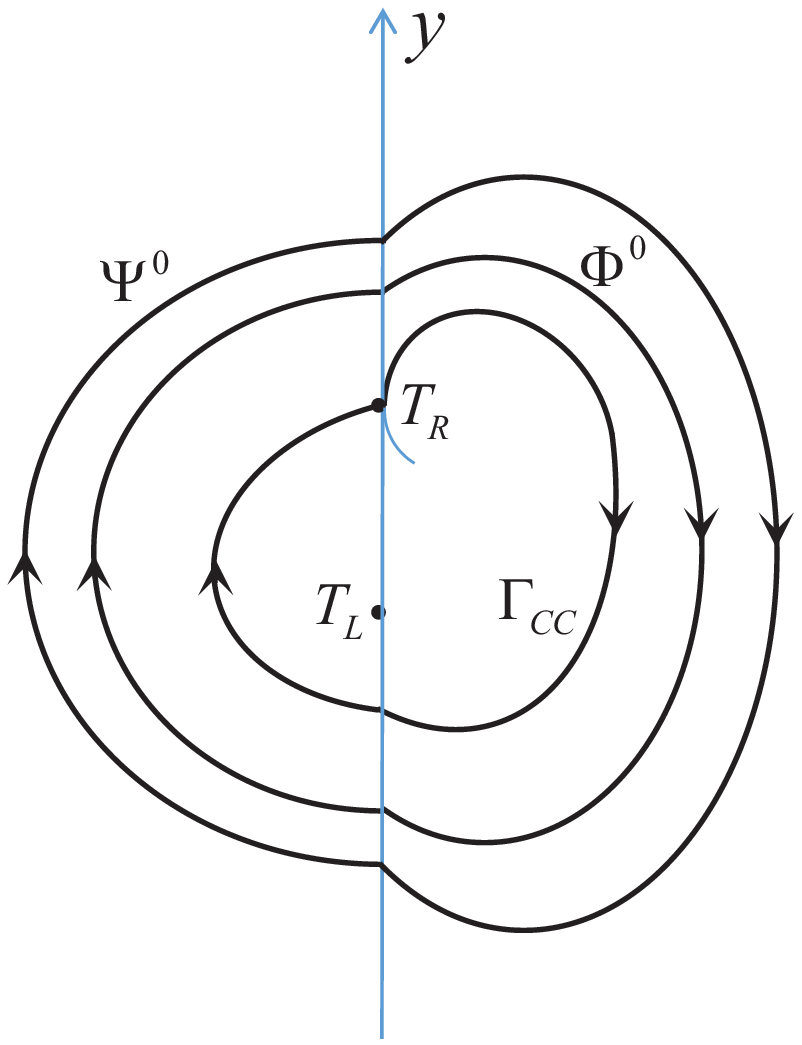}
  \caption*{(b) {\small $\eta=\eta_c(\gamma_1)$}}
  \end{minipage}~~~~~~~~~
  \begin{minipage}[t]{0.25\linewidth}
  \centering
  \includegraphics[width=1.30in]{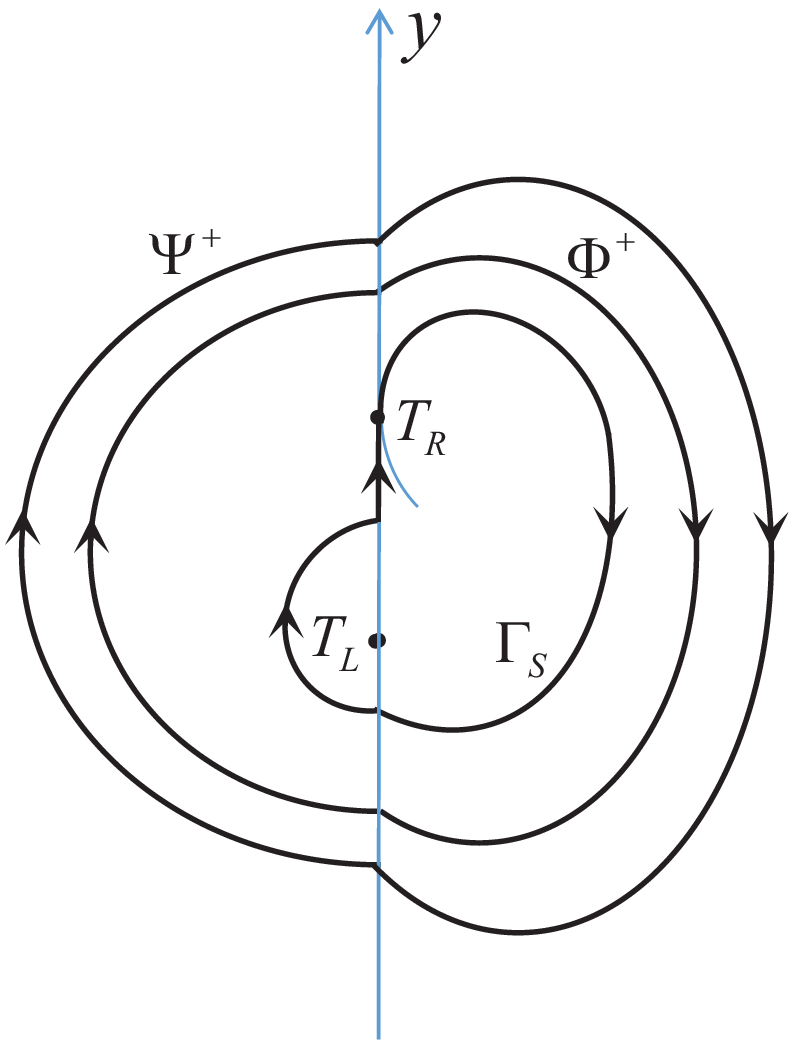}
  \caption*{(c) {\small $\eta_c(\gamma_1)\!\!<\!\eta\!<\!\eta_c(\gamma_1)\!+\!\epsilon_4$}}
  \end{minipage}
\caption{{\small Phase portrait of system (\ref{example2}) with $-2-\epsilon_3<\gamma_1<-2$.}}
\label{EX-2}
\end{figure}

In fact, by the change $(x, y, t)\rightarrow(-x, -y, -t)$ system (\ref{example2}) can be rewritten as
\begin{eqnarray}
\left\{
\begin{aligned}
\dot z&=\left(
\begin{array}{cc}
  -\gamma_1&-1\\
  1+\frac{\gamma_1^2}{4}&0\\
\end{array}
\right)z+
\left(
\begin{array}{c}
  \eta\\
  \frac{(4+\gamma_1^2)(e^{2\pi}-1)}{8}\\
\end{array}
\right)~~~~&&{\rm if}~x>0,\\
\dot z&=\left(
\begin{array}{cc}
  ~-2~&-1~\\
  ~2~&0~\\
\end{array}
\right)z+
\left(
\begin{array}{c}
  0\\
  1\\
\end{array}
\right)~~~~~&&{\rm if}~x<0.
\end{aligned}
\right.
\label{example2c}
\end{eqnarray}
It is easy to check that (\ref{example2c}) satisfies conditions of \cite[Theorem 5(ii)]{ChL-FEEE}, from which
there exist a sufficiently small $\epsilon_3>0$ and a function $\eta_c(\gamma_1)$ such that
(\ref{example2c}) has two structurally stable crossing periodic orbits $\bar\Phi^0$, $\bar\Psi^0$ and a critical crossing periodic orbit $\bar\Gamma_{CC}$ if $-2-\epsilon_3<\gamma_1<-2$ and
$\eta=\eta_c(\gamma_1)$. So we obtain Figure~\ref{EX-2}(b), where $\Gamma_{CC}, \Phi^0, \Psi^0$ correspond to
$\bar\Gamma_{CC}, \bar\Phi^0, \bar\Psi^0$ under transformation $(x, y, t)\rightarrow(-x, -y, -t)$.
Since $\bar\Gamma_{CC}$ is structurally unstable, the crossing-sliding bifurcation happens when $\eta$ varies near $\eta_c(\gamma_1)$.
By \cite[Theorem 5(ii)]{ChL-FEEE} again, there exists a sufficiently small $\epsilon_4>0$ such that the critical crossing periodic orbit
$\bar\Gamma_{CC}$ becomes a structurally stable crossing periodic orbit $\bar\Gamma_C$ if $-2-\epsilon_3<\gamma_1<-2$ and
$\eta_c(\gamma_1)-\epsilon_4<\eta<\eta_c(\gamma_1)$, a sliding periodic orbit $\bar\Gamma_S$ which
can be transformed into Figure~\ref{1A}(b) if $-2-\epsilon_3<\gamma_1<-2$ and $\eta_c(\gamma_1)<\eta<\eta_c(\gamma_1)+\epsilon_4$.
Since $\bar\Phi^0, \bar\Psi^0$ are structurally stable, we obtain Figure~\ref{EX-2}(c) if $\eta_c(\gamma_1)<\eta<\eta_c(\gamma_1)+\epsilon_4$. Hence, this claim is proved and then conclusion (iv) holds.
\end{proof}

In the following, we give some remarks on Theorem~\ref{AC}.

Considering conclusion (ii), we know that system (\ref{DPWL}) has two admissible foci under given conditions.
Thus the uniqueness of crossing periodic orbits is obtained for some parameter regions in the case of two admissible foci.
For system (\ref{DPWL}) of focus-focus type, we also notice that
some sufficient conditions for the uniqueness of crossing periodic orbits are given in \cite{ChL-EE, ChL-FEEE},
but \cite{ChL-EE} is for the case of zero admissible foci and \cite{ChL-FEEE} is for the case of one admissible focus.

Considering conclusion (iii), from its proof we obtain that the inner crossing periodic orbit $\Gamma_{CC}$ is stable and the outer one $\Psi^0$ is unstable. On the other hand, note that $\Gamma_{CC}$ in Figure~\ref{EX-1}(b) is structurally unstable and
the so-called crossing-sliding bifurcation (see \cite{ChL-FEEE1,ChL-YA}) happens when $\rho$ varies near $\rho_c(\alpha)$.
In other word, there exists sufficiently small $\epsilon_2>0$ such that the critical crossing periodic orbit $\Gamma_{CC}$ disappears
and a new sliding periodic orbit $\Gamma_S$ appears  if
$\rho_c(\alpha)<\rho<\rho_c(\alpha)+\epsilon_2$, a new crossing periodic orbit $\Gamma_C$
appears if $\rho_c(\alpha)-\epsilon_2<\rho<\rho_c(\alpha)$. Therefore, we obtain Figure~\ref{EX-1}(a) for the former and Figure~\ref{EX-1}(c)
for the latter due to the structural stability of $\Psi^0$ and $\Phi$. Besides, we observe that in Figure~\ref{EX-1}(a) there are
two sliding periodic orbits which can be transformed into Figure~\ref{2A}(c).
Thus (\ref{example1}) with $0<\alpha<\epsilon_1, \rho_c(\alpha)<\rho<\rho_c(\alpha)+\epsilon_2$
can be regarded as an example to show the existence of Figure~\ref{2A}(c).

Considering conclusion (iv), from its proof we obtain the stability of $\Phi^+$ and $\Psi^+$. In particular,
the inner crossing periodic orbit $\Phi^+$ is unstable and the outer one $\Psi^+$ is stable.

\section{Concluding remarks}
\setcounter{equation}{0}
\setcounter{lm}{0}
\setcounter{thm}{0}
\setcounter{rmk}{0}
\setcounter{df}{0}
\setcounter{cor}{0}

In this paper, for discontinuous piecewise linear system (\ref{DPWL}) we study the number and configuration of sliding periodic orbits, the coexistence of sliding periodic orbits and crossing ones. In this section, we give some concluding remarks to end this paper. For brevity, we
denote the numbers of crossing periodic orbits and sliding periodic orbits by $\mathcal{N}_C$ and $\mathcal{N}_S$, respectively.

In Section 2, we prove that $\mathcal{N}_S\in\{0, 1, 2\}$ in Theorem~\ref{typeAnumber}. Moreover,
we prove that in the sense of $\Sigma$-equivalence and time reversing, the configuration of any sliding periodic orbit is one of these four configurations shown in Figure~\ref{1A} and the configuration of coexistent sliding periodic orbits is one of these three configurations
shown in Figure~\ref{2A}.
Thus, in the cases of  $\mathcal{N}_S=1$ and $\mathcal{N}_S=2$, we show all configurations of
sliding periodic orbits in Figures~\ref{1A} and \ref{2A}, respectively.
In order to show the existence of all possible configurations given in Figures~\ref{1A} and \ref{2A},
in the following we take some examples for each configuration.
Let $\beta_0:=-1/(e^{t^*}\sin t^*)$ and $t^*\in(\pi, 2\pi)$ to satisfy $\cos t^*-\sin t^*-e^{-t^*}=0$.
Then system  (\ref{dsgsh}) has a unique sliding periodic orbit as Figure~\ref{1A}(a)-(d) when \\
(1) $0<\alpha\ll1$, $\beta=\gamma_1=\gamma_2=\gamma_3=\eta=\rho=1,\delta=0$,
\\
(2) $\alpha=\delta=\eta=-\gamma_2=1, \gamma_1=\gamma_3=\rho=0, \beta_0<\beta<2\beta_0$,\\
(3) $\alpha=\delta=\eta=1, \gamma_1=-\gamma_2=2, \gamma_3=0, \beta=\beta_0, 0<\rho+\beta_0\ll 1$,\\
(4) $\alpha=\delta=\eta=1, \gamma_1=-\gamma_2=2, \gamma_3=0, -1\ll \beta-\beta_0<0, 0<\rho+\beta_0\ll 1$,\\
respectively. System  (\ref{dsgsh}) has two sliding periodic orbits as Figure~\ref{2A}(a)-(c) when\\
(5) $\alpha=\beta=-\delta=\eta=\rho=1, -\gamma_1=\gamma_2=2, \gamma_3=0$,\\
(6) $0<\alpha\ll 1, \beta=\delta=\eta=-\rho=1, \gamma_1=2\alpha, \gamma_2=-1-\alpha^2, \gamma_3=0$,\\
(7) $0<\alpha<\epsilon_1, \beta=\delta=\eta=1, \gamma_1=-\gamma_2=2, \gamma_3=0, \rho_c(\alpha)<\rho<\rho_c(\alpha)+\epsilon_2$, \\
respectively.
Here $\epsilon_1, \epsilon_2, \rho_c(\alpha)$ are given in the proof of conclusion (iii) of Theorem~\ref{AC}.
We omit the analysis of these examples because they are similar to the analysis of system (\ref{example1}).

In Section 3, we study the relationship of $\mathcal{N}_S$ and $\mathcal{N}_{C}$.
Here we summarize the type of $(\mathcal{N}_{C}, \mathcal{N}_S)$ in the case $\mathcal{N}_{C}+\mathcal{N}_S>0$.
Let ${\cal T}_{C-S}$ be the set of all types of $(\mathcal{N}_{C}, \mathcal{N}_S)$ in the case $\mathcal{N}_{C}+\mathcal{N}_S>0$.
We claim that
\begin{eqnarray}
\{(0, 1), (1, 1), (2, 1), (0, 2), (1, 2)\}\subseteq {\cal T}_{C-S}.
\label{CSA}
\end{eqnarray}
In fact, the reachability of types $(0, 2), (1,1), (1,2)$, $(2, 1)$ can be obtained directly from Theorem~\ref{AC}.
In example (1), system (\ref{dsgsh}) has a sliding periodic orbit, but it has no crossing periodic orbits because the switching line $\Sigma$ is as shown in Figure~\ref{csl}(b). This implies the reachability of type $(0, 1)$. Thus, (\ref{CSA}) holds.

In many works, the research on the number of periodic orbits of (\ref{DPWL})
only focus on the maximum number of crossing periodic orbits. Moreover, $3$ is the best result as in
\cite{ChL-DC, ChL-CCJ, ChL-EE1, ChL-FEEE, ChL-SX, ChL-LLP, ChL-JLDM, ChL-JLDM1, ChL-LE}.
We have checked in all published articles obtaining three crossing periodic orbits that there exist no sliding periodic orbits
when there are three crossing periodic orbits. Thus, $3$ is also the best result on the maximum number of isolated periodic orbits
in this sense. However, our result provides another viewpoint to obtain three isolate periodic orbits. That is, three isolate periodic
orbits can consist of either two crossing periodic orbits and one sliding periodic orbit or one crossing periodic orbit and
two sliding periodic orbits from (\ref{CSA}).

By Theorem~\ref{AC}, the number of crossing periodic orbits is at most 1 when either $\mathcal{N}_S=2$
or $\mathcal{N}_S=1$ and the unique sliding periodic orbit is as Figure~\ref{1A}(c)(d).
In order to obtain more crossing periodic orbits, we only need to consider either $\mathcal{N}_S=0$ or
$\mathcal{N}_S=1$ and the unique sliding periodic orbit is as Figure~\ref{1A}(a)(b).
By (\ref{CSA}), we get that the maximum number of crossing periodic orbits is at least 2 when the latter holds.



{\footnotesize

\begin{thebibliography}{99}

\bibitem{ChL-AA}
A. A. Andronov, A. A. Vitt, S. E. Khaikin
{\it Theory of Oscillators}, Pergamon Press,
Oxford, 1966.

\bibitem{CO}
A. Colombo, P. Lamiani, L. Benadero, M. di Bernardo,
Two-parameter bifurcation analysis of the buck converter,
{\it SIAM J. Appl. Dyn. Syst.} {\bf 8}(2009), 1507-1522.

\bibitem{MBCJ}
M. di Bernardo, C. J. Budd, A. R. Champneys,
Grazing, skipping and sliding: Analysis of the nonsmooth dynamics of the dc/dc buck converter,
{\it Nonlinearity} {\bf11}(1998), 859-890.

\bibitem{ChL-DB}
M. di Bernardo, C. J. Budd, A. R. Champneys, P. Kowalczyk,
{\it Piecewise-Smooth Dynamical Systems: Theory and Applications},
Applied Mathematical Sciences, Vol.163, Springer Verlag, London, 2008.

\bibitem{ChL-BH}
H. Bilharz,
${\rm \ddot{U}}$ber eine gesteuerte eindimensionale Bewegung,
{\it Z. Angew. Math. Mech.} {\bf 22}(1942), 206-215.

\bibitem{ChL-DC}
D. C. Braga, L. F. Mello,
Limit cycles in a family of discontinuous piecewise linear differential systems with two zones in the plane,
{\it Nonlinear Dyn.} {\bf 73}(2013), 1283-1288.

\bibitem{ChL-CCJ}
C. Buzzi, C. Pessoa, J. Torregrosa,
Piecewise linear perturbations of a linear center,
{\it Discrete Contin. Dyn. Syst.} {\bf 9}(2013), 3915-3936.

\bibitem{HCJL}
H. Chen, J. Llibre, Y. Tang,
Global dynamics of a SD oscillator,
{\it Nonlinear Dyn.} {\bf 91}(2018), 1755-1777.

\bibitem{ChL-AF}
A. F. Filippov,
{\it Differential Equation with Discontinuous Righthand Sides},
Kluwer Academic Publishers, Dordrecht, 1988.

\bibitem{ChL-EE}
E. Freire, E. Ponce, F. Torres,
Canonical discontinuous planar piecewise linear systems,
{\it SIAM J. Appl. Dyn. Syst.} {\bf 11}(2012), 181-211.

\bibitem{ChL-EE1}
E. Freire, E. Ponce, F. Torres,
A general mechanism to generate three limit cycles in planar Filippov systems with two zones,
{\it Nonlinear Dyn.} {\bf 78}(2014), 251-263.

\bibitem{ChL-FEEE}
E. Freire, E. Ponce, F. Torres,
The discontinuous matching of two planar linear foci can have three nested crossing limit cycles,
{\it Publ. Mat.} Vol. extra(2014), 221-253.

\bibitem{ChL-FEEE1}
E. Freire, E. Ponce, F. Torres,
On the critical crossing cycle bifurcation in planar Filippov systems,
{\it J. Differential Equations} {\bf 259}(2015), 7086-7107.

\bibitem{ChL-FG}
F. Giannakopoulos, K. Pliete,
Planar systems of piecewise linear differential equations with a line of discontinuity,
{\it Nonlinearity} {\bf 14}(2001), 1611-1632.

\bibitem{ChL-GF1}
F. Giannakopoulos, K. Pliete,
Closed trajectories in planar relay feedback systems,
{\it Dyn. Syst.} {\bf 17}(2002), 343-358.

\bibitem{ChL-MG}
M. Guardia, T. M. Seara, M. A. Teixeira,
Generic bifurcations of low codimension of planar Filippov systems,
{\it J. Differential Equations} {\bf 250}(2011), 1967-2023.

\bibitem{ChL-JK}
J. K. Hale, {\it Ordinary Differential Equations,}
Wiley-Interscience, New York, 1969.

\bibitem{ChL-HZ}
M. Han, W. Zhang,
On Hopf bifurcation in non-smooth planar systems,
{\it J. Differential Equations} {\bf 248}(2010), 2399-2416.

\bibitem{ChL-SX}
S. Huan, X. Yang,
On the number of limit cycles in general planar piecewise linear systems,
{\it Discrete Contin. Dyn. Syst.} {\bf 32}(2012), 2147-2164.

\bibitem{ChL-SX1}
S. Huan, X. Yang,
Existence of limit cycles in general planar piecewise linear systems of saddle-saddle dynamics,
{\it Nonlin. Anal.} {\bf92}(2013), 82-95.

\bibitem{ChL-SX2}
S. Huan, X. Yang,
On the number of limit cycles in general planar piecewise linear systems of node-node types,
{\it J. Math. Anal. Appl.} {\bf411}(2014), 340-353.

\bibitem{PK}
P. Kowalczyk, P.T. Piiroinen,
Two-parameter sliding bifurcation of periodic solutions in a dry-friction oscillator,
{\it Physica D} {\bf 237}(2008), 1053-1073.

\bibitem{ChL-LLP}
L. Li,
Three crossing limit cycles in planar piecewise linear systems with saddle-focus type,
{\it Electron. J. Qual. Theory Differ. Equ.} (2014), 70: 1-14.

\bibitem{LT}
T. Li, X. Chen, J. Zhao,
Harmonic solutions of a dry friction system,
{\it Nonlin. Anal.: Real World Appl.} {\bf35}(2017), 30-44.

\bibitem{ChL-JLDM}
J. Llibre, D. D. Novaes, M. A. Teixeira,
Limit cycles bifurcating from the periodic orbits of a discontinuous piecewise linear differential center with two zones,
{\it Int. J. Bifur. Chaos} {\bf 25}(2015), 1550144.

\bibitem{ChL-JLDM1}
J. Llibre, D. D. Novaes, M. A. Teixeira,
On the birth of limit cycles for non-smooth dynamical systems,
{\it Bull. Sci. Math.} {\bf 139} (2015), 229-244.

\bibitem{ChL-LE}
J. Llibre, E. Ponce,
Three nested limit cycles in discontinuous piecewise linear differential systems with two zeros,
{\it Dyn. Contin. Discrete Implus. Syst. Ser. B Appl. Algorithm} {\bf 19}(2012), 325-335.

\bibitem{ChL-ZXL}
J. Llibre, X. Zhang,
Limit cycles for discontinuous planar piecewise linear differential systems separated by one straight and having a center,
{\it J. Math. Anal. Appl.} {\bf467}(2018), 537-549.

\bibitem{ChL-YA}
Yu. A. Kuznetsov, S. Rinaldi, A. Gragnani,
One-parameter bifurcations in planar Filippov systems,
{\it Int. J. Bifur. Chaos} {\bf 13}(2003), 2157-2188.

\bibitem{ChL-PK}
K. Pliete,
${\rm \ddot{U}}$ber die Anzahl geschlossener Orbits bei unstetigen st${\rm \ddot{U}}$ckweise linearen dynamischen Systemen
in der Ebene Diploma Thesis Mathematisches Institut, Universit${\rm \ddot{a}}$t zu K${\rm \ddot{o}}$ln, 1998.

\bibitem{ChL-SS}
S. Shui, X. Zhang, J. Li,
The qualitative analysis of a class of planar Filippov systems,
{\it Nonlinear Anal.} {\bf 73}(2010), 1277-1288.

\bibitem{ChL-HLH1}
J. Wang, X. Chen, L. Huang,
The number and stability of limit cycles for planar piecewise linear systems of node-saddle type,
{\it J. Math. Anal. Appl.} {\bf469}(2019), 405-427.

\bibitem{ChL-HLH}
J. Wang, C. Huang, L. Huang,
Discontinuity-induced limit cycles in a general planar piecewise linear system of saddle-focus type,
{\it Nonlinear Anal.: Hybrid Syst.} {\bf33}(2019), 162-178.
\end{thebibliography}}
\end{document}